\documentclass[11pt]{amsart}
\usepackage{amssymb,amsmath,amsthm, longtable}
\usepackage{enumerate, url, verbatim}
\usepackage{comment}
\usepackage{xy, color}
\usepackage[normalem]{ulem}
\xyoption{all}
\newcommand{\leg}[2]{\genfrac{(}{)}{}{}{#1}{#2}}

\newtheorem{theorem}{Theorem}
\newtheorem{lemma}[theorem]{Lemma}
\newtheorem{corollary}[theorem]{Corollary}

\newtheorem{proposition}[theorem]{Proposition}

\newtheorem{definition}[theorem]{Definition}
\newtheorem{remarks}[theorem]{Remarks}
\newtheorem{remark}[theorem]{Remark}

\newtheorem{example}[theorem]{Example}

\numberwithin{theorem}{section} \numberwithin{equation}{section}

\renewcommand{\i}{\iota}

\newcommand{\al}{\alpha}
\newcommand{\be}{\beta}
\newcommand{\ga}{\gamma}
\renewcommand{\th}{\theta}
\newcommand{\eps}{\varepsilon}
\newcommand{\LR}{\longrightarrow}
\newcommand{\gothr}{{\mathfrak r}}
\newcommand{\gd}{{\mathfrak d}}

\newcommand{\mfa}{\mathfrak{a}}

\newcommand{\bmu}{\mbox{\boldmath$\mu$}}
\newcommand{\ov}[1]{\overline{#1}}

\newcommand{\isom}{\simeq}

\newcommand{\calD}{\mathcal{D}}

\newcommand{\om}{\omega}
\newcommand{\mfp}{\mathfrak{p}}
\newcommand{\mfb}{\mathfrak{b}}
\newcommand{\mfq}{\mathfrak{q}}
\newcommand{\mff}{\mathfrak{f}}
\newcommand{\f}{\mathfrak{f}}

\newcommand{\calB}{\mathcal{B}}
\newcommand{\calF}{\mathcal{F}}
\newcommand{\calFL}{\mathcal{F}_{\ell}}

\newcommand{\F}{\mathbb{F}}
\newcommand{\Cl}{{\text {\rm Cl}}}
\newcommand{\Tr}{{\text {\rm Tr}}}
\newcommand{\rk}{{\text {\rm rk}}}

\newcommand{\Q}{\mathbb{Q}}

\newcommand{\Z}{\mathbb{Z}}

\newcommand{\GL}{{\text {\rm GL}}}

\newcommand{\p}{\mathfrak p}
\renewcommand{\a}{\mathfrak a}
\newcommand{\q}{\mathfrak q}
\renewcommand{\c}{\mathfrak c}
\renewcommand{\d}{\mathfrak d}
\newcommand{\m}{{\mathfrak m}}
\newcommand{\N}{{\mathcal N}}
\newcommand{\NC}{{\lceil{\mathcal N}\rceil}}

\DeclareMathOperator{\Hom}{Hom}

\newcommand{\Gal}{{\text {\rm Gal}}}

\newcommand{\Res}{\textnormal{Res}}
\newcommand{\Disc}{\textnormal{Disc}}

\renewcommand{\b}{{\mathfrak b}}

\newcommand{\WL}{\widetilde{L}}

\newcommand{\z}{\zeta}
\newcommand{\ze}{\zeta_\ell}

\newcommand*{\longhookrightarrow}{\ensuremath{\lhook\joinrel\relbar\joinrel\rightarrow}}

\begin{document}
\title{On $D_{\ell}$-Extensions of Odd Prime Degree $\ell$}

\author{Henri Cohen}
\address{Universit\'e de Bordeaux, Institut de Math\'ematiques, U.M.R. 5251 du
C.N.R.S, 351 Cours de la Lib\'eration, 33405 TALENCE Cedex, FRANCE}
\email{Henri.Cohen@math.u-bordeaux1.fr}
\author{Frank Thorne}
\address{Department of Mathematics, University of South Carolina, 1523 Greene Street, Columbia, SC 29208, USA}
\email{thorne@math.sc.edu}
\date{\today}

\begin{abstract}
Generalizing the work of A.~Morra and the authors, we give explicit
formulas for the Dirichlet series generating function of $D_{\ell}$-extensions
of odd prime degree $\ell$ with given quadratic resolvent.
Over the course of our proof, we explain connections between our formulas and 
the Ankeny-Artin-Chowla conjecture, the Ohno--Nakagawa relation
for binary cubic forms, and other topics. We also obtain improved upper bounds
for the number of such extensions (over $\Q$) of bounded discriminant.
\end{abstract}
\maketitle

\section{Introduction}

The theory of cubic number fields 
is, in many respects, well understood. One reason for this is that the {\itshape Delone-Faddeev} \cite{DF} and {\itshape Davenport-Heilbronn \cite{DH}}
{\itshape correspondences} parametrize cubic fields in terms of binary cubic forms, 
up to equivalence by an action of $\GL_2(\Z)$, and satisfying certain local conditions. Therefore questions about counting cubic fields
can be reduced to questions about counting lattice points, and this idea has 
led to asymptotic density theorems as well as other interesting results.

In more recent work, Bhargava \cite{B4, B5} obtained similar parametrization and counting results for $S_4$-quartic and $S_5$-quintic fields. However,
generalizing this work to number fields of arbitrary degree $\ell$ seems difficult, if not impossible: the parametrizations of 
$S_3$-cubic, $S_4$-quartic, and $S_5$-quintic fields are all by {\itshape prehomogeneous vector spaces}, and for higher degree fields there is
no apparent prehomogeneous vector
space for which one could hope to establish a parametrization theorem.

In \cite{Coh-Mor} and \cite{Coh-Tho}, A.~Morra and the authors contributed to 
the cubic theory by giving  explicit formulas for the
Dirichlet generating series of discriminants of cubic fields having
\emph{given} resolvent. For example, 
%writing
%\begin{equation}
%\Phi_{-107}(s) = \sum_{\substack{[K : \Q] = 3 \\ \Disc(K) = -107n^2}} n^{-s}\;,
%\end{equation}
we have the explicit formula
\begin{equation}\label{eqn:sample_s3}
\Phi_{-107}(s) := \sum_{\substack{[K : \Q] = 3 \\ \Disc(K) = -107n^2}} n^{-s} = - \dfrac{1}{2} + \dfrac{1}{2}\left(1+\dfrac{2}{3^{2s}}\right)\prod_{\leg{321}{p}=1}\left(1+\dfrac{2}{p^s}\right)\\
+\left(1+\dfrac{2}{3^{2s}}\right)\prod_p\left(1+\dfrac{\om(p)}{p^s}\right),
\end{equation}
where $\om(p)$ is equal to $2$ or $-1$ if $p$ is totally split or inert in the unique cubic field of discriminant $321$, determined by the polynomial 
$x^3-x^2-4x+1$, and $\om(p) = 0$ otherwise. Similar formulas hold when $-107$ is replaced by any other fundamental discriminant $D$;
the formula has one main term, and one additional Euler product for each 
cubic field of discriminant $-D/3$, $-3D$, and $-27D$.

The proofs involve class field theory and Kummer theory; see also work of Bhargava and Shnidman \cite{BS} obtaining related results
through a study of binary cubic forms.

\smallskip
The object of the present paper is to generalize the theory developed in \cite{Coh-Mor} and \cite{Coh-Tho} to degree $\ell$ extensions having Galois group $D_\ell$,
for any odd prime $\ell$. (See also \cite{CoDiOl3, CDO_D4, 
CDOsurvey, maki, wright}, among other references,
for further related results.)

Let $L/k$ be an extension\footnote{A remark on our choice of notation: Readers familiar with \cite{Coh-Tho} or \cite{Coh-Mor} should note that by and large
we adopt the notation of \cite{Coh-Tho} and the progression of 
\cite{Coh-Mor}; the reader knowledgeable with the latter paper can 
immediately see the similarities and differences. (See also Morra's thesis \cite{Mor}
for a version of \cite{Coh-Mor} with more detailed proofs.)
What was called $(K_2,K,L,K'_2)$ in \cite{Coh-Mor} will now be called $(K,L,K_z,K')$ (so that the 
main number field in which most computations take place is $K_z$), and the
field names $(k,k_z,N,N_z)$ stay unchanged. The primitive cube root of
unity $\rho$ is replaced by a primitive $\ell$th root of unity $\ze$.}
of odd prime degree $\ell$, let $N=\WL$ 
be a Galois closure of $L$, and assume that $\Gal(N/k)\isom D_{\ell}$, 
the dihedral group with $2\ell$ elements. We will refer to any such $L$ as a $D_\ell$-{\itshape extension} 
of $k$, or a $D_\ell$-{\itshape field} when $k=\Q$. Below we also refer to $F_\ell$-extensions with the analogous meaning.

There exists a unique quadratic subextension $K/k$ of $N/k$, called the
\emph{quadratic resolvent} of $L$, with
$\Gal(N/K) \isom C_\ell$, and a nontrivial theorem of J.~Martinet
involving the computation of higher ramification groups (see Propositions
10.1.25 and 10.1.28 of \cite{Coh1}) tells us that its conductor $f(N/K)$ is of the form 
$f(N/K)=f(L)\Z_K$, where $f(L)$ is an ideal of the base field $k$, and that the 
relative discriminant $\gd(L/k)$ of $L/k$ is given by the formula 
$\gd(L/k)=\gd(K/k)^{(\ell-1)/2}f(L)^{\ell-1}$. 

We study the set $\calFL(K)$ of $D_{\ell}$-extensions
of $k$ whose quadratic resolvent field is isomorphic to $K$.
(Here and in the sequel, extensions are always considered up to 
$k$-isomorphism.) More precisely, we want to compute as explicitly as possible
the Dirichlet series\footnote{The series also depends on the base field $k$, which we do not include explicitly
in the notation.}
$$\Phi_{\ell}(K,s)=\dfrac{1}{\ell-1}+\sum_{L\in\calFL(K)}\dfrac{1}{\N(f(L))^s}\;,$$
where $\N(f(L))=\N_{k/\Q}(f(L))$ is the absolute norm of the ideal $f(L)$. 

\smallskip

Our most general result is Theorem \ref{thm:main}, which we specialize
to a more explicit version in the case $k=\Q$ as Theorem \ref{thm_main_q}.
This should be considered as the most important result of this paper. In Section
\ref{sec_trans} we prove that our formulas can always be brought into a form
similar to \eqref{eqn:sample_s3}. For example, we have

%Two sample results are as follows:
\begin{equation*}
  \Phi_5(\Q(\sqrt{5}),s)=\dfrac{1}{20}\left(1+\dfrac{4}{5^s}\right)\prod_{p\equiv1\pmod5}\left(1+\dfrac{4}{p^s}\right)+
  \dfrac{1}{5}\left(1-\dfrac{1}{5^s}\right)\prod_{p\equiv1\pmod5}\left(1+\dfrac{\om_E(p)}{p^s}\right)\;,
\end{equation*}
where $E$ is the field defined by $x^5+5x^3+5x-1=0$ of discriminant $5^7$, and
$\om_E(p)=-1$, $4$, or $0$ according to whether $p$ is inert, totally split,
or other in $E$. 

%We also have
%\smallskip
%\begin{equation}\label{eq513}\Phi_5(\Q(\sqrt{13}),s)=\dfrac{1}{20}\left(1+\dfrac{4}{25^s}\right)\prod_p\left(1+\dfrac{4}{p^s}\right)+\dfrac{1}{5}\left(1-\dfrac{1}{25^s}\right)\prod_p\left(1+\dfrac{\omega_E(p)}{p^s}\right)\;,
%\end{equation}
%where the products are over suitable primes $p$ (see Example \ref{ex:intro}), $E$ is the field defined by
%$x^5+5x^3+5x-3=0$, and $\om_E(p)$ is as before.

\medskip

In a companion paper, joint with Rubinstein-Salzedo \cite{CRST}, we investigate a curious twist to this story. Taking the $n = 1$ term of 
formula \eqref{eqn:sample_s3} (or, rather, its generalization to any $D$) yields the nontrivial identity 
\begin{equation}\label{eqn:nakagawa}
N_3(D^*)+N_3(-27D)=\begin{cases}
N_3(D) & \text{ if $D<0$,}\\
3N_3(D)+1 & \text{ if $D>0$,}\end{cases}
\end{equation}
for any fundamental discriminant $D$,
where 
$D^* = -3D$ if $3 \nmid D$ and $D^* = -D/3$ if $3 \mid D$. (Here $N_3(k)$ is the number of cubic fields of discriminant $k$. Note that there are no cubic fields of discriminant $-3D$ if $3 \mid D$.)
This identity was previously conjectured by Ohno \cite{ohno} and then proved by Nakagawa \cite{N},
as a consequence of an `extra functional equation' for the Shintani zeta function associated to the lattice of binary cubic forms. Our generalization
of \eqref{eqn:sample_s3} thus subsumes the Ohno--Nakagawa theorem \eqref{eqn:nakagawa}.

Our proof there {\itshape used} the Ohno--Nakagawa theorem, but in \cite{CRST} we further develop some of the techniques of this paper (in particular, of Section \ref{sec:gb})
to give another proof of \eqref{eqn:nakagawa} and give a generalization to any prime $\ell \geq 3$. For $\ell > 3$ our work
relates counts
of $D_\ell$-fields (the right-hand side of \eqref{eqn:nakagawa}) to counts of $F_\ell$-fields $\ell$ (the left-hand side), where $F_\ell$ is the {\itshape Frobenius
group} of order $\ell (\ell - 1)$, whose definition is recalled in Section \ref{sec_trans}. (Note that $S_3 = D_3 = F_3$.)
The result involves a technical (Galois theoretic) condition on the $F_\ell$-fields which is not automatically satisfied for $\ell > 3$, and
we defer to \cite{CRST} for a complete statement of the results.
It is however important to note that, as for the cubic case, even the case
$n=1$ of our Dirichlet series identities gives
%such as \eqref{eq513} gives
interesting results: for instance, for any negative fundamental discriminant
$-D$ coprime to $5$, we have
\begin{equation}\label{eqn:main_cor_neg}
N_{F_5}\big((-1)^0 5^3 D^2\big)
+ 
N_{F_5}\big((-1)^0 5^5 D^2\big)
+
N_{F_5}\big((-1)^0 5^7 D^2\big)
= 
N_{D_5}\big((-D)^2\big)
+
N_{D_5}\big((-5 D)^2 \big)\;,
\end{equation}
and if instead $D > 1$, then we have
\begin{equation}\label{eqn:main_cor_pos}
N_{F_5}\big((-1)^2 5^3 D^2\big)
+ 
N_{F_5}\big((-1)^2 5^5 D^2\big)
+
N_{F_5}\big((-1)^2 5^7 D^2\big)
=
5 \Big( N_{D_5}\big(D^2\big)
+
N_{D_5}\big((5D)^2 \big) \Big) + 2\;.
\end{equation}
(Here, $N_G((-1)^r X)$ denotes the number of quintic fields with discriminant exactly equal to $X$, with
$r$ pairs of complex embeddings, and whose Galois closure has Galois group $G$ over $\Q$.)

If we want an identity counting $D_5$-fields of discriminant $(\pm D)^2$ {\itshape or}
$(\pm 5D)^2$ alone, then the left side of \eqref{eqn:main_cor_neg} 
and \eqref{eqn:main_cor_pos}
becomes more complicated, and involves the Galois condition mentioned above. The relevance to the present paper is
that it is {\itshape precisely those $F_\ell$-fields counted by this
identity that yield Euler products.} We describe this in more detail in Section \ref{sec_trans}.

\medskip

There is one further curiosity that emerges in our work: a connection to a well-known conjecture attributed to\footnote{Ankeny,
Artin, and Chowla did not conjecture this in \cite{AAC}, although they did explicitly ask if it is true.
 Mordell \cite{Mordell}
attributed the conjecture to them in followup work, where he proved the conjecture for regular primes.}
Ankeny, Artin, and Chowla \cite{AAC}
which states that 
if $\ell \equiv 1 \pmod 4$ is prime and $\epsilon = (a + b \sqrt{\ell})/2$ is
the fundamental unit of $\Q(\sqrt{\ell})$, then $\ell \nmid b$.
As we will see, the truth or falsity of the conjecture will be reflected in our explicit formula for $D_{\ell}$-extensions having
quadratic resolvent $\Q(\sqrt{\ell})$. Note that the conjecture is known to 
be true for $\ell<2\cdot 10^{11}$ (see \cite{PRW}), but on heuristic grounds it should be
false: if we assume independence of the divisibility by $\ell$, the number
of counterexamples for $\ell\le X$ should be around $\log(\log(X))/2$;
in addition, numerous counterexamples can easily be found for ``fake''
quadratic fields, see e.g., \cite{Coh5, Oh}. 

\medskip
Separately, we obtain improved bounds for the number $N_\ell(D_\ell, X)$ of $D_\ell$-fields $L$
with $|\Disc(L)| < X$:
\begin{theorem}\label{thm:improved_bounds}
We have 
\begin{equation}
N_\ell(D_\ell, X) \ll_{\ell, \epsilon} X^{\frac{3}{\ell - 1} - \frac{1}{\ell (\ell - 1)} + \epsilon}.
\end{equation}
\end{theorem}
This improves on Kl\"uners's \cite{kluners} bound of $O(X^{\frac{3}{\ell - 1} + \epsilon})$.
The proof (in Section \ref{sec:upper_bounds}) is independent of the rest of the paper, and is an
immediate consquence of applying Ellenberg, Pierce, and Wood's \cite{EPW} recent
bounds on $\ell$-torsion in quadratic fields
within Kl\"uners's method.

\medskip
Our work follows several other papers studying dihedral field extensions. Much of the theory (such as Martinet's theorem)
is described in the first author's book \cite{Coh1}. Another reference is 
Jensen and Yui \cite{JY}, who
studied $D_{\ell}$-extensions from multiple points of view. They proved that if $\ell \equiv 1 \pmod 4$ is a regular prime,
then no $D_{\ell}$-extension of $\Q$ has discriminant a power of $\ell$, and we will
recover and strengthen their result. Jensen and Yui also studied the problem of constructing $D_{\ell}$-extensions, and gave several examples.

Another relevant work is the paper of Louboutin, Park, and Lefeuvre \cite{LPL}, who developed a general class field
theory method to construct real $D_{\ell}$-extensions. These problems have also been addressed in the function field
setting by Weir, Scheidler, and Howe \cite{WSH}.

\bigskip
Since some of the proofs are quite technical, we give a detailed overview 
of the contents of this paper.

We begin in Section \ref{sec:galois} with a characterization of
the fields $L\in\calFL(K)$ using Galois and Kummer theory. These fields are in bijection with 
elements of
$K_z := K(\ze)$ modulo $\ell$th powers, satisfying certain restrictions which
guarantee that the associated Kummer extensions of $K_z$ descend to degree
$\ell$ extensions of $k$. Writing such an extension as $K_z(\root\ell\of\al)$ with
$\alpha \Z_{K_z} = \prod_{0 \leq i \leq \ell - 2} \mfa_i^{g^i} \mfq^\ell$, we further characterize these fields
in terms of conditions on the $\mfa_i$ and an associated member $\ov{u}$ of a Selmer group associated to $K_z$.

These conditions are described in terms of the group ring $\mathbb{F}_\ell[G]$, where $G = \Gal(K_z/k)$.
Groups such as $K_z^*/{K_z^*}^\ell$, $\Cl(K_z)[\ell]$, and the Selmer group are naturally
$\mathbb{F}_\ell[G]$-modules, and our conditions correspond to being annihilated
by certain elements of $\mathbb{F}_\ell[G]$ (see Definition \ref{defT}). 

In Section \ref{sec:galois} we also study the subfields of $K_z/k$, with particular attention to a degree $\ell - 1$ extension
$K'/k$
called the {\itshape mirror field} of $K$; we will see that much of the arithmetic of prime splitting in various
extensions can be conveniently expressed in terms of $K'$. 

The reader who is willing to take our technical computations for granted is
advised to look only at the necessary definitions in the intermediate sections
and to skip directly to Section \ref{sec:semifinal}.

In Section \ref{sec:conductors}, we give an expression for the `conductor' $f(L)$ in terms of the quantities $\mfa_i$ and
$\ov{u}$ defined in Section \ref{sec:galois}. The main result, Theorem \ref{thmfnk}, was proved by the first author,
Diaz y Diaz, and Olivier in \cite{CoDiOl3} in their study of cyclic extensions of degree $\ell$.
% and we also prove a few
%additional related lemmas and propositions. 
Unfortunately the results of that section are rather complicated to state,
and oblige us to introduce a fair amount of notation.

In Section \ref{sec:dir_series} we begin to study the fundamental Dirichlet series using the results proved in Section
\ref{sec:conductors}, and in Section \ref{sec:comp_gb}
we study the size of a certain Selmer group appearing in our formulas. The latter section is heavily algebraic and again
appeals heavily to the results of \cite{CoDiOl3}.

In Section \ref{sec:semifinal}, we put everything together to obtain our most general formula (Theorem \ref{thm:main})
for $\Phi_{\ell}(K, s)$, a generalization of the main theorem of \cite{Coh-Mor}. Subsequently we work to 
make everything more explicit,
for the most part specialized to the case $k = \Q$. In Section \ref{sec_q} we compute various quantities appearing
in Theorem \ref{thm:main} for $k = \Q$, leading to Theorem \ref{thm_main_q}, a more explicit specialized version of Theorem 
\ref{thm:main}. We also obtain asymptotics for counting $D_\ell$-extensions of $\Q$, proved in 
Corollary \ref{corasym}.

The formula of Theorem \ref{thm_main_q} involves a somewhat complicated group $G_{\mfb}$, and in Section \ref{sec:gb}
we further study its size. The main result is the Kummer pairing of 
Theorem \ref{thm:main_kummer}, familiar from (for example) the proof of the Scholz reflection principle, and fairly 
simple to prove. One important input (Proposition \ref{prop:kum_hec}) is a very nice 
relationship, due essentially to Kummer and Hecke, 
between the conductor of Kummer extensions of $K_z$, and congruence properties of the $\ell$th roots used to generate them.

%Some of our work in Section \ref{sec:gb} (including Theorem \ref{thm:main_kummer})
%is also critical in \cite{CRST}, and to avoid redundancy we only sketch a few results whose
%complete proofs are given there.

In Section \ref{sec:gb} we also explore the connection to the conjecture of Ankeny, Artin, and Chowla mentioned above. The truth or falsity of this conjecture will then be reflected in our explicit formula (Proposition \ref{prop:special_explicit}) 
for $\Phi_{\ell}(\Q(\sqrt{\ell}), s)$, and in Corollary \ref{cor:ac_exist} we will give a proof of an observation of Lemmermeyer,
that the existence of $D_{\ell}$-fields ramified only at $\ell$ is equivalent to the falsity of the Ankeny-Artin-Chowla conjecture.

In Section \ref{sec_trans} we further study the characters of the group $G_{\mfb}$, and prove (in Theorem \ref{thm:char_nf})
that each such character corresponds to an $F_\ell$-extension $E/k$, such that the values of $\chi$ correspond to the splitting
types of primes in $E$. This was done for $\ell = 3$ and $k = \Q$ in \cite{Coh-Tho}, but in Theorem \ref{thm:char_nf} we do not
require $k = \Q$.

It is here that the connection to the Ohno--Nakagawa theorem emerges; for $\ell = 3$ and $k = \Q$, we established in \cite{Coh-Tho}
(using Ohno--Nakagawa) that the set of characters of $G_{\mfb}$ corresponds precisely to a suitable and easily described set of
fields $E$. For $\ell > 3$ we require the generalization of Ohno--Nakagawa established in \cite{CRST}, and so in Section
\ref{sec_trans} we say a bit more about the results of \cite{CRST} and explain their relevance. We also prove an explicit
formula valid for the `special case' $K = \Q(\sqrt{\ell})$.

\section*{Acknowledgements}
We would like to thank Michael Filaseta, David Harvey, Franz Lemmermeyer, Hendrik Lenstra, David Roberts, and John Voight for helpful comments
and suggestions.

This material is based upon work supported by the National Science Foundation under Grant No. DMS-1201330 and
by the National Security Agency under a Young Investigator Grant.

\section{Galois and Kummer Theory}\label{sec:galois}

\subsection{Galois and Kummer theory, and the Group Ring}

We will use the results of \cite{CoDiOl3}, but before stating them we need 
some notation. We denote by $\ze$ a primitive $\ell$th root of unity,
we set $K_z=K(\ze)$, $k_z=k(\ze)$, $N_z=N(\ze)$, and we denote by $\tau$, $\tau_2$, and $\sigma$ 
generators of $k_z/k$, $K/k$, and $N/K$ respectively, with $\tau^{\ell - 1} = \tau_2^2 = \sigma^{\ell} = 1$.

The number $\ze$ could belong to $k$, or to $K$, or generate a nontrivial 
extension of $K$ of degree dividing $\ell-1$. These essentially correspond 
respectively to cases (3), (4), and (5) of \cite{Coh-Mor} (cases (1) and (2)
correspond to cyclic extensions of $k$ of degree $\ell$, which have been
treated in \cite{CoDiOl3}). Cases (3) and (4) are 
considerably simpler since we do not have to adjoin $\ze$ to $K$ to
apply Kummer theory. 

We are particularly interested in the case $k=\Q$, in which case 
either $[K_z:K]=\ell-1$, or $[K_z:K]=(\ell-1)/2$, i.e., $K\subset k_z$, which 
is equivalent to $K=\Q(\sqrt{\ell^*})$ with $\ell^*=(-1)^{(\ell-1)/2}\ell$. To balance
generality and simplicity, {\itshape we assume that $k$ is any number field for which
$[k_z:k]=\ell-1$.} Then, as for $k = \Q$ there are two possible cases:
either $[K_z:K]=\ell-1$, which we call the \emph{general case}, or 
$K\subset k_z=K_z$ and $[K_z:K]=(\ell-1)/2$, which we will call the
\emph{special case}. Note that if $\ell=3$ this 
means that $\ze\in K$, so we are in case (4), but there is no reason to treat
this case separately. It should not be particularly difficult to extend our results to any base field $k$,
as was done in \cite{CoDiOl3}.

We set the following notation:

\begin{itemize}
\item We let $g$ be a primitive root modulo $\ell$, and 
also denote by $g$ its image in $\F_\ell^*=(\Z/\ell\Z)^*$.
\item We let $G=\Gal(K_z/k)$. Thus in the general case 
$G\isom(\Z/2\Z)\times(\Z/\ell\Z)^*$, while in the special case 
$G=\Gal(k_z/k)\isom(\Z/\ell\Z)^*$. We denote by $\tau$ the unique element of 
$\Gal(k_z/k)$ such that $\tau(\ze)=\ze^g$, so that $\tau$ generates
$\Gal(k_z/k)$, and we again denote by $\tau$ its lift to $K_z$ or $N_z$.
\end{itemize}

The composite extension $N_z = N K_z$ is Galois over $k$, and $\sigma$ and 
$\tau$ naturally lift to $N_z$. In the general case, $\tau$ 
and $\sigma$ commute; in the special case, $\tau^2$ is a generator of 
$\Gal(K_z/K)$ and $\tau_2$ can be taken to be any odd power of $\tau$, 
for instance $\tau$ itself, so that $\tau\sigma\tau^{-1}=\sigma^{-1}$.

This information is summarized in the two Hasse diagrams below, depicting the general
and special cases respectively.

$$
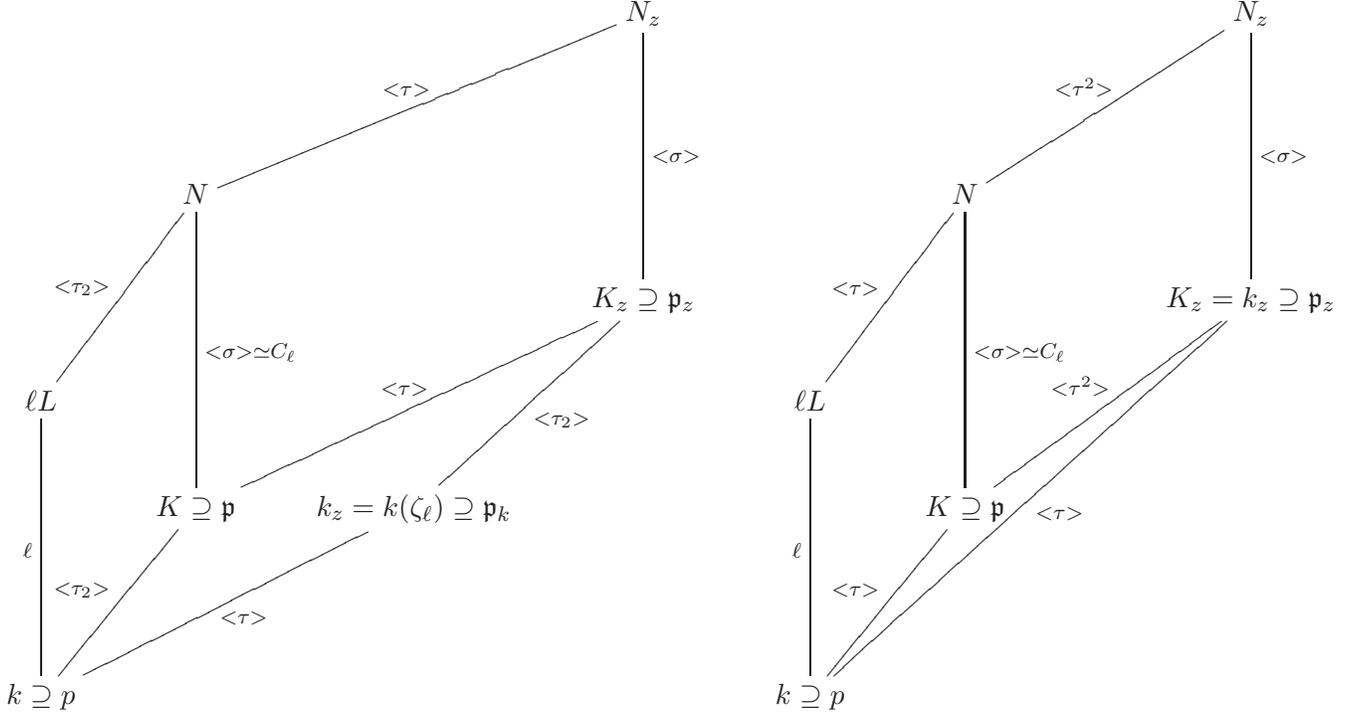
\xymatrix{
&&&N_z\ar@{-}_{<\tau>}[ddll]\ar@{-}^{<\sigma>}[ddd]&
&&&N_z\ar@{-}_{<\tau^2>}[ddll]\ar@{-}^{<\sigma>}[ddd]\\
&&&&&
&&&&\\
&N\ar@{-}_{<\tau_2>}[ddl]\ar@{-}^{<\sigma>\isom C_\ell}[ddd]&&&
&N\ar@{-}_{<\tau>}[ddl]\ar@{-}^{<\sigma>\isom C_\ell}[ddd]&\\
&&&K_z \supseteq \p_z\ar@{-}_{<\tau>}[ddll]\ar@{-}^{<\tau_2>}[ddl]&
&&&K_z=k_z \supseteq \p_z \ar@{-}_{<\tau^2>}[ddll]\ar@{-}^{<\tau>}[ddddlll]\\
\ell L\ar@{-}_\ell[ddd]&&&&
\ell L\ar@{-}_\ell[ddd]&&&\\
&K\supseteq \p\ar@{-}_{<\tau_2>}[ddl]&k_z=k(\ze) \supseteq \p_k \ar@{-}^{<\tau>}[ddll]&&
&K\supseteq \p\ar@{-}_{<\tau>}[ddl]&&&\\
&&&&&
&&&&\\
k \supseteq p&&&&
k \supseteq p&&&}$$

In the above $p, \p, \p_k, \p_z$ indicate our typical notation (to be used later)
for primes of $k, K, k_z, K_z$ respectively. We write $\ell L$ to indicate the existence of $\ell$
fields isomorphic to $L$.

\begin{lemma}\label{lemidem} For $a\bmod(\ell-1)$ and $b\bmod2$, set
$$e_a=\dfrac{1}{\ell-1}\sum_{j\bmod(\ell-1)}g^{aj}\tau^{-j}\in\F_\ell[G]\text{\quad and, in the general case,\quad}
e_{2,b}=\dfrac{1}{2}\sum_{j\bmod2}(-1)^{bj}\tau_2^{-j}\;.$$
The $e_a$ form a complete set of orthogonal idempotents in $\F_\ell[G]$, as do
the $e_{2,b}$ in the general case, so in the general case any 
$\F_\ell[G]$-module $M$ has a canonical decomposition
$M=\sum_{a\bmod(\ell-1),\ b\bmod2}e_ae_{2,b}M$, while in the special case we 
simply have $M=\sum_{a\bmod(\ell-1)}e_aM$.
\end{lemma}

\begin{proof} Immediate and classical; see, e.g., Section 7.3 of \cite{Eis}.\end{proof}

We set the following definitions: 

\begin{definition}\label{defT}
In the group ring $\F_\ell[G]$, we set 
$$T=\begin{cases}
\{\tau_2+1,\tau-g\}\text{\quad in the general case\;,}\\
\{\tau+g\}\text{\quad in the special case\;.}\end{cases}$$
\begin{enumerate}
\item We define $\i(\tau_2+1)=e_{2,1} = \frac{1}{2}(1 - \tau_2)$, and for any $a$
we define $\i(\tau-g^a)=e_a$, so that for instance
$\i(\tau+g)=e_{(\ell+1)/2}$.
\item For any $\F_\ell[G]$ module $M$, we denote by 
$M[T]$ the subgroup annihilated by all the elements of $T$.
\end{enumerate}\end{definition}

\begin{lemma}\label{lem24} Let $M$ be an $\F_\ell[G]$-module.
\begin{enumerate}
\item For any $t\in T$ we have $t\circ\i(t)=\i(t)\circ t=0$, where the
action of $t$ and $\i(t)$ is on $M$.
\item For all $t\in T$ we have $M[t]=\i(t)M$ and $M[\i(t)]=t(M)$.
\item If $x\in M[t]$ then $\i(t)(x)=x$.
\end{enumerate}
\end{lemma}

\begin{proof} This follows from Lemma \ref{lemidem}. In particular, $\tau e_a = g^a e_a$,
so that the image of $\tau - g^a$ is $\sum_{b \neq a} e_a M$.
\end{proof}
\subsection{The Bijections}

\begin{proposition}\label{propbij1}\begin{enumerate}
\item There exists a bijection between elements $L\in\calFL(K)$ and classes
of elements $\ov{\al}\in(K_z^*/{K_z^*}^\ell)[T]$ such that $\ov{\al}\ne\ov{1}$,
modulo the equivalence relation identifying $\ov{\al}$ with $\ov{\al^j}$ for
all $j$ with $1\le j\le\ell-1$.
\item If $\al\in K_z^*$ is some representative of $\ov{\al}$, the extension
$L/k$ corresponding to $\al$ is the field $K_z(\root\ell\of\al)^G$, i.e.,
the fixed field of $K_z(\root\ell\of\al)$ by a lift of  $G=\Gal(K_z/k)$ to
$\Gal(K_z(\root\ell\of\al)/k)$.
\end{enumerate}\end{proposition}

\begin{proof}
First assume that $L \in \calFL(K)$;
% and define $K_z$, $N$, $N_z$ as above;
since $\ze\in K_z$, by Kummer theory cyclic extensions of
degree $\ell$ of $K_z$ are of the form $N_z=K_z(\root\ell\of\al)$, where 
$\ov{\al}\ne\ov{1}$ is unique in $K_z^*/{K_z^*}^\ell$ modulo the equivalence
relation mentioned in the proposition. As $N_z$ determines $L$ up to conjugacy,
we must prove that $\ov{\al}$ is annihilated by $T$.

Writing $N_z = K_z(\th)$ with 
$\th^\ell=\al$,
we may assume the generator $\sigma$ chosen so that $\sigma(\th)=\ze\th$.
Set $\eps=1$ if we are in the general case, $\eps=-1$ if we are in the
special case, so that $\tau\sigma\tau^{-1}=\sigma^\eps$. We have 
$\tau(\ze)=\ze^g$, so that
$$\sigma(\tau(\th))=\tau(\sigma^\eps(\th))=\tau(\ze^\eps)\tau(\th)=\ze^{\eps g}\tau(\th)\;,$$
hence if we set $\eta=\tau(\th)/\th^{\eps g}$ we have
$\sigma(\eta)=\ze^{\eps g}\tau(\th)/\ze^{\eps g}\th^{\eps g}=\eta$. It follows
by Galois theory that $\eta\in K_z$, so that 
$\tau(\al)/\al^{\eps g}=\eta^\ell\in{K_z^*}^\ell$, hence
that $\ov{\al}\in (K_z^*/{K_z^*}^\ell)[\tau-\eps g]$. 

Concerning $\tau_2$ (in the general case only), 
the relation $\tau_2\sigma\tau_2^{-1}=\sigma^{-1}$ similarly
shows that $\sigma(\th\tau_2(\th))=\th\tau_2(\th)$ so that
$\ov{\al}\in (K_z^*/{K_z^*}^\ell)[\tau_2+1]$, in other words
$\ov{\al}\in (K_z^*/{K_z^*}^\ell)[T]$ as desired.

\medskip
To conclude, we must prove that each $\ov{\al}\in (K_z^*/{K_z^*}^\ell)[T]$ determines
such an $L \in \calFL(K)$. Again write $\theta = \root\ell\of\al$ with $\sigma(\theta) = \zeta_{\ell} \theta$ and $N_z = K_z(\theta)$.
Define an automorphism $\overline{\tau}$ of $N_z$, agreeing with $\tau$ on $K_z$, by writing $\overline{\tau}(\theta) = \eta \theta^{\eps g}$
($\eps = \pm 1$ as before), where 
$\eta^{\ell} = \tau(\alpha)/ \alpha^{ \eps g} \in K_z^{\ell}$, so that 
$\eta \in K_z$ is well-defined up to an $\ell$th root
of unity, and we make an arbitrary such choice.

Computations show that $\ov{\tau} \sigma^{\eps} (\theta) = \sigma \ov{\tau} (\theta)$ and that
$\ov{\tau}^{\ell - 1}(\theta)$ is $\theta$ times a root of unity. Each $\ov{\tau} \sigma^i$ is also a lift
of $\tau$. 

In the general case, 
we check that there is a unique such lift, which we denote simply by $\tau$, for 
which $\tau^{\ell - 1}(\theta) = \theta$. 
Write $\ov{\tau_2} (\theta) = \eta_2 / \theta$ with 
$\eta_2^{\ell} = \alpha \tau_2(\alpha)$ where $\eta_2$ is in 
$K_z$ and indeed $k_z$.
We check that $\ov{\tau_2}^2(\theta) = \theta$ and
$\ov{\tau_2} \sigma(\theta) = \sigma^{-1} \ov{\tau_2}$, so that by rewriting $\ov{\tau_2}$ as $\tau_2$
we see that $N_z/k$ is Galois with Galois group $C_{\ell - 1} \times D_{\ell}$, as required.
Here the choice of lift $\tau_2$ is not uniquely determined: $D_{\ell}$ has
$\ell$ elements of order $2$, corresponding to the $\ell$ conjugate subextensions $L/k$
of degree $\ell$.

In the special case, rewriting $\overline{\tau}$ as $\tau$ we now have $\tau^{\ell - 1} = 1$
regardless of the choice of lift: we have $\tau^{\ell - 1}(\theta) = \zeta_\ell^i \theta$ for some
integer $i$, so that unless $i \equiv 0 \pmod \ell$, $\tau$ is of order $\ell (\ell - 1)$. We already know
that $N_z/k$ is Galois, as the $\tau^r \sigma^s$ are distinct automorphisms of $N_z/k$ for
$0 \leq \tau < \ell - 1, \ 0 \leq \sigma < \ell$. We have already proved that $\Gal(N_z/k)$ is nonabelian,
and in particular noncyclic, hence $i = 0$. So $\tau^{\ell - 1} = 1$ and $\Gal(N_z/k)$ has the
required presentation.

\end{proof}

Recall from \cite{Coh1} the following definition:

\begin{definition} We denote by $V_\ell(K_z)$ the group of
\emph{($\ell$-)virtual units} of $K_z$, in other words the group of
$u\in K_z^*$ such that $u\Z_{K_z}=\q^\ell$ for some ideal $\q$ of
$K_z$, or equivalently such that $\ell\mid v_{\p_z}(u)$ for any prime ideal
$\p_z$ of $K_z$. We define the \emph{($\ell$-)Selmer group} $S_\ell(K_z)$ of
$K_z$ by $S_\ell(K_z)=V_\ell(K_z)/{K_z^*}^\ell$.\end{definition}

The following lemma shows in particular that the Selmer group is finite.
\begin{lemma}\label{lemseq} We have a split exact sequence of 
$\F_\ell[G]$-modules 
$$1\LR\dfrac{U(K_z)}{U(K_z)^\ell}\LR S_\ell(K_z)\LR \Cl(K_z)[\ell]\LR1\;,$$
where the last nontrivial map sends $\ov{u}$ to the ideal class of $\q$ such
that $u\Z_{K_z}=\q^\ell$.
\end{lemma}
\begin{proof}
Exactness follows from the definitions, and the sequence splits because $\ell \nmid |G|$
(see for example \cite[Lemma 3.1]{CoDiOl1} for a proof).
\end{proof}

From Lemma \ref{lem24} we extract the following technical result.

\begin{lemma}\label{lem:bij_tech}
Given $t \in T$ and $\alpha \in K_z^*$ such that $t(\alpha)$ is a virtual 
unit, we have $t(\alpha) = \gamma^{\ell} t(u)$ for some $\gamma \in K_z^*$ and
some virtual unit $u$.

Moreover, if $\alpha$ is annihilated modulo ${K_z^*}^{\ell}$ by
$t' \neq t  \in T$, we may choose $u$ to be annihilated by $t'$ in 
$S_{\ell}(K_z)$.
\end{lemma}

\begin{proof}
Given $t$ and $\alpha$, (1) of Lemma \ref{lem24} applied to 
$M = K_z^*/{K_z^*}^\ell$ implies that $\iota(t)(t(\alpha)) \in {K_z^*}^{\ell}$.
Since $t(\alpha)$ is a virtual unit, its image $\overline{t(\alpha)}$ is 
annihilated by $\iota(t)$ in the Selmer group. By Lemma \ref{lem24} applied to
$M = S_{\ell}(K_z)$, we have $\overline{t(\alpha)} = t(\overline{\beta})$ for 
some $\overline{\beta} \in S_{\ell}(K_z)$, giving the first result. For the 
second, we replace each of the modules $M$ by $M[t']$: since $t$ and $t'$ 
commute, if $\alpha \in M$ is annhilated by $t'$, so is $t(\alpha)$. 
\end{proof}

\begin{proposition}\label{propbij2}\begin{enumerate}
\item There exists a bijection between elements $L\in\calFL(K)$ and equivalence
classes of $\ell$-tuples $(\a_0,\dots,\a_{\ell-2},\ov{u})$ modulo the
equivalence relation 
$$(\a_0,\dots,\a_{\ell-2},\ov{u})\sim (\a_{-i},\dots,\a_{\ell-2-i},\ov{u^{g^i}})$$
for all $i$ (with the indices of the ideals $\a$ considered modulo $\ell-1$),
where the $\a_i$ and $\ov{u}$ are as follows:
\begin{enumerate}
\item The $\a_i$ are coprime integral squarefree ideals of $K_z$ such that if 
we set $\a=\prod_{0\le i\le \ell-2}\a_i^{g^i}$ then the ideal class of $\a$ 
belongs to $\Cl(K_z)^\ell$, and $\ov{\a}\in (I(K_z)/I(K_z)^\ell)[T]$, where as
usual $I(K_z)$ denotes the group of (nonzero) fractional ideals of $K_z$.
\item $\ov{u}\in S_{\ell}(K_z)[T]$, and in addition $\ov{u}\ne\ov{1}$ when
$\a_i=\Z_{K_z}$ for all $i$.
\end{enumerate}
\item Given $(\a_0,\dots,\a_{\ell-2})$, $\a$, and $\overline{u}$ as in (a), the field $L \in \calFL(K)$ is determined as follows:
There exist an ideal $\q_0$ and an element $\al_0\in K_z$ 
such that $\a\q_0^\ell=\al_0\Z_{K_z}$ with 
$\ov{\al_0}\in (K^*_z/{K^*_z}^\ell)[T]$. Then $L$ is any of the $\ell$ conjugate
degree $\ell$ subextensions of $N_z=K_z(\root\ell\of{\al_0u})$, where $u$ is an arbitrary lift of 
$\ov{u}$.\end{enumerate}\end{proposition}

\begin{proof} Given $L$, associate $N_z=K_z(\root\ell\of\al)$ as in Proposition \ref{propbij1}. We may write uniquely
$\al\Z_{K_z}=\prod_{0\le i\le \ell-2}\a_i^{g^i}\q^{\ell}$, where the
$\a_i$ are coprime integral squarefree ideals of $K_z$, and they must satisfy the conditions of (a).

Each $\a$ which thus occurs satisfies $\a \q^\ell = \al_0 \Z_{K_z}$ for some
$\al_0$ with $\ov{\al_0} \in (K^*_z/{K^*_z}^\ell)[T]$, and for each $\a$ we arbitrarily associate such an $\al_0$.
Given $\a \q^\ell = \al \Z_{K_z}$, $u := \al / \al_0$ is a virtual unit; writing $\overline{u}$ for its class
in $S_\ell(K_z)$, $\overline{u}$ is annhiliated by $T$ because both $\overline{\al}$ and $\overline{\al_0}$ are.

This establishes the bijection, and we conclude by observing the following:
\begin{itemize}
\item The elements
$\al$ and $\be$ give
equivalent extensions if and only if $\be=\al^{g^i}\ga^\ell$ for some
element $\ga$ and some $i$ modulo $\ell-1$, and then if 
$\al_0\Z_{K_z}=\prod_j\a_j^{g^j}\q^\ell$ and $\al=\al_0u$, we have on the
one hand $\be\Z_{K_z}=\prod_j\a_{j-i}^{g^j}\q_1^\ell$ for some ideal $\q_1$,
so the ideals $\a_j$ are permuted cyclically, and on the other hand
$\be=(\al_0u)^{g^i}\ga^{\ell}=\al_0^{g^i}u^{g^i}\ga^\ell$, so $\ov{u}$
is changed into $\ov{u}^{g^i}$, giving the equivalence described in (1).
\item
The only fixed point of the transformation
$(\a_0,\dots,a_{\ell-2},\ov{u})\mapsto(\a_{\ell-2},\a_0,\dots,\a_{\ell-3},\ov{u^g})$
is obtained with all the $\a_i$ equal and $\ov{u}=\ov{u^g}$, but since 
the $\a_i$ are pairwise coprime this means that they are all equal to
$\Z_{K_z}$, and $\ov{u}=\ov{u^{g^i}}$ for all $i$, and so $\ov{u}=1$.
\end{itemize}
\end{proof}

\begin{remark}\label{remclk} Note that condition (a) implies that $\ov{\a}\in(\Cl(K_z)/\Cl(K_z)^\ell)[T]$, and for any 
modulus $\m$ coprime to $\a$ also that
$\ov{\a}\in(\Cl_{\m}(K_z)/\Cl_{\m}(K_z)^\ell)[T]$.
\end{remark}

\begin{lemma}\label{lemconda1} Keep the above notation, and in 
particular recall that $\a=\prod_{0\le i\le\ell-2}\a_i^{g^i}$. The 
condition $\ov{\a}\in (I(K_z)/I(K_z)^\ell)[T]$ is equivalent to the 
following:
\begin{enumerate}\item In the general case $\tau(\a_i)=\a_{i-1}$ 
(equivalently, $\a_i=\tau^{-i}(\a_0)$), and 
$\tau^{(\ell-1)/2}(\a_0)=\tau_2(\a_0)$.
\item In the special case $\tau(\a_i)=\a_{i+(\ell-3)/2}$, so that
$\a_{2i}=\tau^{-2i}(\a_0)$ and $\a_{2i+1}=\tau^{-2i}(\a_1)$, with
the following conditions on $(\a_0,\a_1)$:
\begin{itemize}\item If $\ell\equiv1\pmod4$ then
$\a_1=\tau^{(\ell-3)/2}(\a_0)$, or equivalently 
$\a_0=\tau^{(\ell+1)/2}(\a_1)$.
\item If $\ell\equiv3\pmod4$ then $\tau^{(\ell-1)/2}(\a_0)=\a_0$
and $\tau^{(\ell-1)/2}(\a_1)=\a_1$.
\end{itemize}\end{enumerate}
\end{lemma}

\begin{proof} Since $\tau(\a)=\prod_i\tau(\a_i)^{g^i}$ and the 
$\tau(\a_i)$ are integral, squarefree and coprime ideals, this is the 
canonical decomposition of $\tau(\a)$ (up to $\ell$th powers). On the 
other hand $\a^g=\prod_i\a_{i-1}^{g^i}$. Assume first that we are in
the general case. Since $\tau(\a)/\a^g$ is an $\ell$th power,
by uniqueness of the decomposition we deduce that $\tau(\a_i)=\a_{i-1}$.
A similar proof using that $g^{(\ell-1)/2}\equiv-1\pmod\ell$ shows that
$\tau_2(\a_i)=\a_{i+(\ell-1)/2}$, and putting everything together proves
(1). Assume now that we are in the special case, so that 
$\tau(\a)/\a^{-g}$ is an $\ell$th power. Since 
$-g\equiv g^{(\ell+1)/2}\pmod{\ell}$, the same reasoning shows that
$\tau(\a_i)=\a_{i-(\ell+1)/2}=\a_{i+(\ell-3)/2}$, so in particular
$\tau^2(\a_i)=\a_{i-(\ell+1)}=\a_{i-2}$, and the other formulas
follow immediately.\end{proof}

\begin{definition}\label{defcald} Let $\calD$ (resp., $\calD_\ell$) 
be the set of all prime ideals $p$ of $k$ with $p\nmid\ell$ (resp., with
$p\mid\ell$) such that the prime ideals $\p$, $\p_z$, and (in the general case)
$\p_k$
of $K$, $K_z$, and $k_z$  above $p$
satisfy the following conditions:

\begin{enumerate}\item In all cases $\p$ is totally split in the extension 
$K_z/K$.
\item In the general case $\p_k$ is split in the quadratic extension 
$K_z/k_z$.
\item In the special case with $\ell\equiv1\pmod4$, $p$ is totally split in the extension
$K_z/k$ (equivalently $p$ is split in the quadratic extension $K/k$).
\end{enumerate}

\end{definition}

Note that these conditions are independent of the choices of $\p$, $\p_z$, and $\p_k$ above
any particular $p$.

\begin{corollary}\label{corsplit1} If $\p_z$ is a prime ideal of $K_z$ 
dividing some $\a_i$, above a prime $p$ of $k$, then $p \in \calD \cap \calD_\ell$.
\end{corollary}

\begin{proof} Assume first that we are in the general case. Then $\tau$ acts transitively
on the $\mfa_i$, all of which are squarefree and coprime, and so any $\mfp$ dividing $\mfa_i$ must have
$\ell - 1$ nontrivial conjugates (including $\mfp$ itself), establishing (1). Similarly, 
$\tau_2(\a_i)=\a_{i+(\ell-1)/2}$, and for the same reason 
the prime ideals of $K_z$ dividing the $\a_i$ 
come from prime ideals $\p_k$ of $k_z$ which split in $K_z/k_z$.

In the special case, if $p$ splits as a product of
$h$ conjugate ideals in $K_z$, the decomposition group 
$D(\p_z/p)$ has cardinality $ef=(\ell-1)/h$ hence is the subgroup of 
$\Gal(K_z/k)$ generated by $\tau^h$ since $[K_z:k]=\ell-1$. Since 
$\tau^h(\a_i)=\a_{(\ell-3)h/2+i}$ and $\tau^h$ fixes $\p_z$, it follows as 
before that $(\ell-1)\mid(\ell-3)h/2$. Now evidently $(\ell-1,(\ell-3)/2)$
is equal to $1$ if $\ell\equiv1\pmod4$ and to $2$ if $\ell\equiv3\pmod4$. Thus
when $\ell\equiv1\pmod4$ we deduce as above that $(\ell-1)\mid h$ hence that 
$e=f=1$, so that $p$ is totally split in $K_z/k$. On the other hand if 
$\ell\equiv3\pmod4$ we only have $(\ell-1)/2\mid h$. If $h=\ell-1$ then $p$ is
again totally split. On the other hand, if $h=(\ell-1)/2$ then $ef=2$, so
$p$ is either inert or ramified in the quadratic extension $K/k$, so
$\p$ is totally split in $K_z/K$.\end{proof}

\subsection{The Mirror Field}
We now introduce the {\itshape mirror field} of $K$. When $\ell = 3$ this notion is
classical and well-known; the mirror field of $\Q(\sqrt{D})$ is $\Q(\sqrt{-3D})$ and the
{\itshape Scholz reflection principle} establishes that the $3$-ranks of their class groups
differ by at most $1$. 

In the case $\ell > 3$ this notion is less well known but 
does appear in the literature (see for instance the works
of G.~Gras \cite{Gra1} and \cite{Gra2}), and in
particular Scholz's theorem can be generalized to this context, see for instance
\cite{Kis} for the case $\ell=5$.

\begin{definition} In the general case, we define the \emph{mirror field}
$K'$ of $K$ (implicitly, with respect to the prime $\ell$) to be the 
degree $\ell-1$ subextension of $K_z/k$ fixed by $\tau^{(\ell-1)/2}\tau_2$.
\end{definition}

We do not define the mirror field for the special case, so in this subsection we assume that we are in the general case.

\begin{lemma}\label{lemmirror} Write $K=k(\sqrt{D})$ for some $D\in k^*\setminus {k^*}^2$.
\begin{enumerate}
\item The extension $K'/k$ is cyclic of degree $\ell-1$, and $K'=k\bigl(\sqrt{D}(\ze-\ze^{-1})\bigr)$.
\item The field $K'$ is a quadratic extension of $k(\ze+\ze^{-1})$,
more precisely $$K'=k(\ze+\ze^{-1})\Bigl(\sqrt{-D(4-\al^2)}\Bigr)\;,$$ 
where $\al=\ze+\ze^{-1}$. 
\end{enumerate}
\end{lemma}

\begin{proof} Straightforward; for (2), note that 
$-D(4-\alpha^2) = D(\ze-\ze^{-1})^2$.
\end{proof}

The point of introducing the mirror field is the following result:

\begin{proposition}\label{propkp} Assume that we are in the general case.
As before, let $p$ be a prime ideal of $k$, $\p_z$ an ideal of $K_z$ above 
$p$, and $\p_k$ and $\p$ the prime ideals below $\p_z$ in $k_z$ and $K$ 
respectively. The following are equivalent:
\begin{enumerate}\item The ideals $\p_k$ and $\p$ are both totally split in 
$K_z/k_z$ and $K_z/K$ respectively (in other words $p\in\calD\cup\calD_\ell$).
\item The ideal $p$ is totally split in $K'/k$.
\end{enumerate}
In particular (by Corollary \ref{corsplit1}), (1)-(2) are true if $\p_z$ divides some $\a_i$.
Moreover, these conditions imply that exactly one of the following is true:
\begin{enumerate}[\ \ \ \ (a)]\item $p$ is split in $K/k$ and totally split in $k_z/k$.
\item $p$ is inert in $K/k$ and split in $k_z/k$ as a product of $(\ell-1)/2$
prime ideals of degree $2$.
\item $p$ is above $\ell$, is ramified in $K/k$, and its absolute
ramification index $e(p/\ell)$ is an odd multiple of $(\ell-1)/2$ (equivalently
$e(\p/\ell)$ is an odd multiple of $\ell-1$).
\end{enumerate}
\end{proposition}

\begin{proof} (1) if and only if (2): We see that any nontrivial elements of
$D(\p_z/p)$ must be of the form $\tau^i\tau_2$ with $i \not \equiv 0 \pmod{\ell - 1}$, 
and squaring we have $\tau^{2i}\in D(\p_z/p)$, so $2i\equiv0\pmod{\ell-1}$,
so $D(\p_z/p)\subset\{1,\tau^{(\ell-1)/2}\tau_2\}$, yielding (2).  
The converse is proved similarly.

\smallskip

To prove the last statement, first recall from \cite{CoDiOl1} the following 
result:

\begin{lemma}\label{condkz} Let $K$ be any number field and $K_z=K(\ze)$. The 
conductor of the extension $K_z/K$ is given by the formula
$$\f(K_z/K)=\prod_{\substack{\p\mid\ell\\(\ell-1)\nmid e(\p/\ell)}}\p\;.$$
\end{lemma}

It follows in particular that if $p\nmid\ell$, or if $p\mid\ell$ and 
$(\ell-1)\mid e(p/\ell)$ then $p$ is unramified in $k_z/k$, and therefore
also (arguing via inertia groups) in $K/k$, since otherwise the ideal $\p_k$ would be 
ramified in $K_z/k_z$. Thus, assuming (2), the only prime ideals $p$ which can
be ramified in $K/k$ are with $p\mid\ell$ and $(\ell-1)\nmid e(p/\ell)$.

If $p$ is split or inert in $K/k$, we check that $f(\p_z | p)$ equals $1$ or $2$ respectively, showing (a) and (b).
If $p$ is ramified, then \eqref{eq:pzp} implies that
$(\ell-1)\mid e(\p/\ell)=e(\p/p)e(p/\ell)$. Since
$(\ell-1)\nmid e(p/\ell)$ we conclude that $e(p/\ell)=n(\ell-1)/2$ with $n$ odd.
\end{proof}

The following corollaries are immediate:
\begin{corollary}\label{ramell} Let $p$ be a prime ideal of $k$ below a prime
ideal $\p_z$ of $K_z$ dividing some $\a_i$ as defined above. If $p$ is ramified in the quadratic 
extension $K/k$ then $p$ is above $\ell$.\end{corollary} 

\begin{corollary}\label{corsimp} In both the general and special cases, assume
that for any prime ideal $p$ of $k$ above $\ell$ the absolute ramification
index $e(p/\ell)$ is not divisible by $(\ell-1)/2$. Then all the
$\a_i$ are coprime to $\ell$.
\end{corollary}

Note that for $\ell=3$ this corollary is empty, but the conclusion of the corollary always holds
when $\ell>2[k:\Q]+1$, and in particular when $k=\Q$ and $\ell\ge5$.

\begin{proposition}\label{propsplit2} There exists an ideal $\a_{\al}$ of $K$
such that $\prod_{0\le i\le\ell-2}\a_i=\a_{\al}\Z_{K_z}$. In addition:
\begin{enumerate}
\item In the general (resp., special) case, $\a_{\al}$ is stable by $\tau$ and $\tau_2$ 
(resp., by $\tau$).
\item If either the assumption of Corollary \ref{corsimp} is satisfied (for 
instance when $\ell>2[k:\Q]+1$), or we are in the special case with 
$\ell\equiv1\pmod4$, then $\a_{\al}=\a'_{\al}\Z_K$ for some ideal $\a'_{\al}$
of $k$.\end{enumerate}\end{proposition}

\begin{proof} (1). In the general case, since $\tau(\a_i)=\a_{i-1}$ we 
have $\prod_{0\le i\le\ell-2}\a_i=\a_{\al}\Z_{K_z}$ with 
$\a_{\al}=\N_{K_z/K}(\a_0)$, and since $\tau_2(\a_i)=\a_{i+(\ell-1)/2}$,
$\a_{\al}$ is stable by $\tau_2$. In the special case, since
$\tau^2(\a_i)=\a_{i-2}$ we have 
$\prod_{0\le i<(\ell-1)/2}\a_{2i}=\N_{K_z/K}(\a_0)\Z_{K_z}$ and
$\prod_{0\le i<(\ell-1)/2}\a_{2i+1}=\N_{K_z/K}(\a_1)\Z_{K_z}$, so that
$\prod_{0\le i<\ell-1}\a_i=\a_{\al}\Z_{K_z}$ with
$\mfa_{\al}=\N_{K_z/K}(\a_0\a_1)$ an ideal of $K$, and since
$\tau(\a_i)=\a_{i+(\ell-3)/2}$, $\a_{\al}$ is stable by $\tau$.

\smallskip

(2). In the special case with $\ell\equiv1\pmod4$ then $(\ell-3)/2$
is odd, so since $\a_1=\tau^{(\ell-3)/2}(\a_0)$ it follows that 
$\tau(\N_{K_z/K}(\a_0))=\N_{K_z/K}(\a_1)$, so that
$\prod_{0\le i\le\ell-2}\a_i=\N_{K_z/k}(\a_0)\Z_{K_z}=\a'_{\al}\Z_{K_z}$ with
$\a'_{\al}$ an ideal of the base field $k$. On the other hand, if the
assumption of Corollary \ref{corsimp} is satisfied then $\a_{\al}$ is coprime to
$\ell$, hence by Corollary \ref{ramell} it is not divisible by any prime
ramified in $K/k$, and since it is stable by $\Gal(K/k)$ it comes from an
ideal $\a'_{\al}$ of $k$.\end{proof}

\section{Hecke Theory: Conductors}\label{sec:conductors}

Our goal (see Theorem \ref{thmfnk}) 
is to give a usable expression for the ``conductor'' $f(L)$
in terms of the fundamental quantities $(\a_0,\cdots,\a_{\ell-2},\ov{u})$
given by Proposition \ref{propbij2}, where we recall that the conductor of the
$C_\ell$-extension $N/K$ is equal to $f(N/K)=f(L)\Z_K$ and that 
$\gd(L/k)=\gd(K/k)^{(\ell-1)/2}f(L)^{\ell-1}$.

%In this section we will denote by $p$ a prime ideal of $k$ over $\ell$, by $\p$ a prime 
%ideal of $K$ above $p$, by $\p_z$ a prime ideal of $K_z$ above $\p$, and in the general case, by $\p_k$ a prime ideal of $k_z$ below $\p_z$.

We first recall from \cite{CoDiOl1} and \cite{CoDiOl3} some results concerning
the cyclotomic extensions $k_z/k$ and $K_z/K$. 

\begin{remark} By and large we stick to the notation of \cite{CoDiOl3}
  except that the notation $m(\mfp)$ of \cite{CoDiOl3} is the same as $M(\mfp)$
  here, which corresponds to numbers $A_{\al}$, while our $m(\mfp)$
  corresponds to numbers $a_{\al}$.
\end{remark}

\begin{proposition}\cite[Theorem 2.1]{CoDiOl3}\label{proppzp} As above, let $p$ be a prime of $k$ over $\ell$,
and let $e(\p)$ and $e(p)$ be the respective absolute ramification indices
over $\ell$. Then we have
\begin{equation}\label{eq:pzp}
e(\p_z/\p)=\dfrac{\ell-1}{(\ell-1,e(\p))}\text{\quad and\quad}\dfrac{e(\p_z/\ell)}{\ell-1}=\dfrac{e(\p)}{(\ell-1,e(\p))}\;.
\end{equation}
\end{proposition}

\begin{definition}\label{defpowfrac}
Suppose that $p$, $\p$, and $\p_z$ are as above, 
so that $e(\p_z/p)\mid(\ell-1)$. Moreover, 
let $\ov{\al}\in (K_z^*/{K_z^*}^\ell)[T]$ be as in Proposition \ref{propbij1}.
\begin{enumerate}
\item If $p\Z_K=\p^2$ in $K/k$ we set $p^{1/2}=\p$, and if
$\p\Z_{K_z}=\p_z^{e(\p_z/\p)}$ in $K_z/K$,
we set $\p_z=\p^{1/e(\p_z/\p)}$.
\item We say that an ideal $p$ of $k$ divides some $\Gal(K_z/k)$-invariant ideal $\b$ of
$K$ (resp., of $K_z$) when $(p\Z_K)^{1/e(\p/p)}$ (resp.,
$(p\Z_K)^{1/e(\p_z/p)}$) does, or equivalently when $\p$ (resp.,
$\p^{1/e(\p_z/\p)}$) does, where this last condition is independent of the choice of 
ideal $\p$ of $K$ above $p$.
\item If $e$ is an integer, write $r(e)$ for the unique integer
such that $e\equiv r(e)\pmod{\ell-1}$ and $1\le r(e)\le\ell-1$.
\item We write
$$M(\p)=\frac{\ell e(\p_z/\ell)}{\ell-1} = \frac{ \ell e(\p)}{(\ell - 1, e(\p))} \in \Z \ , \ \ \
m(\p)=\frac{M(\p)}{e(\p_z/\p)}=\frac{\ell e(\p)}{\ell-1}\;.$$
\item 
Denote by $E_n$ the congruence $x^\ell/\al\equiv1\pmod{^*\p_z^n}$ in $K_z$.
\item Define quantities $A_{\al}(p)$ and $a_{\al}(p)$ as follows:
\begin{itemize}
\item If $E_n$ is soluble for $n=M(\p)$, we set 
$A_{\al}(\p)=M(\p)+1$ and $a_{\al}(\p)=m(\p)$.
\item Otherwise, if $n<M(\p)$ is the largest exponent for 
which it is soluble, we set $A_{\al}(\p)=n$ and we define
$$a_{\al}(\p)=\dfrac{A_{\al}(\p)-r(e(\p))/(\ell-1,e(\p))}{e(\p_z/\p)}
= 
\left\lceil\dfrac{A_{\al}(\p)}{e(\p_z/\p)}\right\rceil-1 \in \Z \;.$$
\end{itemize}
\end{enumerate}
\end{definition}

\begin{remarks}\begin{enumerate}
\item The quantity $r(e(\p))/(\ell-1,e(\p))=r(e(\p))/(\ell-1,r(e(\p)))$
is an integer, and equals $1$ when $\ell=3$ or when
$k=\Q$ for instance, and the second equality for $a_\al(\p)$ is proved below.
\item The notation $A_{\al}(\p)$ and $a_{\al}(\p)$ (instead of $A_{\al}(\p_z)$ 
and $a_{\al}(\p_z)$) is justified by the following lemma:
\end{enumerate}\end{remarks}

\begin{lemma}\label{lemsol} With the above assumptions, the solubility of $E_n$
%(the congruence $x^\ell/\al\equiv1\pmod{^*\p_z^n}$) 
is independent of the ideal 
$\p_z$ of $K_z$ above $p$; i.e.,
it is equivalent to 
$x^\ell/\al\equiv1\pmod{^*p^{n/e(\p_z/p)}}$ or to
$x^\ell/\al\equiv1\pmod{^*\p^{n/e(\p_z/\p)}}$.
\end{lemma}

\begin{proof} If $\p'_z$ is another ideal above $p$, there exists
$h=\tau^i\tau_2^j\in \Gal(K_z/k)$ with $\p'_z=h(\p_z)$ 
(resp., simply $h=\tau^i$ in the
special case). Thus if $x^\ell/\al\equiv1\pmod{^*\p_z^n}$ we have 
$h(x)^\ell/h(\al)\equiv1\pmod{^*{\p'_z}^n}$. However, since
$\ov{\al}\in (K_z^*/{K_z^*}^\ell)[T]$, modulo $\ell$th powers we have 
$\ov{\tau(\al)}=\ov{\al^g}$ and $\ov{\tau_2(\al)}=\ov{\al^{-1}}$
(resp., $\ov{\tau(\al)}=\ov{\al^{-g}}$), hence
$h(\al)=\al^{(-1)^jg^i}\ga^\ell$ (resp., $h(\al)=\al^{(-1)^ig^i}\ga^\ell$)
for some $\ga\in K_z^*$. We deduce that
$y^\ell/\al\equiv1\pmod{^*{\p'_z}^n}$, with
$y=(h(x)/\ga)^{(-1)^jg^{-i}}$ (resp.,
$y=(h(x)/\ga)^{(-1)^ig^{-i}}$),
proving the lemma.
\end{proof}

\begin{proposition}\label{propap}\begin{enumerate}
\item We have $\ell\nmid A_{\al}(\p)$, and if $A_{\al}(\p)\leq M(\p)$ (equivalently, if $A_{\al}(\p)\leq M(\p) - 1$) 
then
$$A_{\al}(\p)\equiv \dfrac{e(\p)}{(\ell-1,e(\p))} \ \left( \textnormal{mod} \  \dfrac{\ell-1}{(\ell-1,e(\p))} \right)\;.$$
\item We have $a_{\al}(\p)=m(\p)$ if 
$A_{\al}(\p)=M(\p)+1$, and otherwise 
$$0\le a_{\al}(\p)\le \dfrac{\ell e(\p)}{\ell-1}-\dfrac{\ell-1+r(e(\p))}{\ell-1}
<\dfrac{\ell e(\p)}{\ell-1}-1 = m(\p) - 1\;.$$
\end{enumerate}\end{proposition}
\begin{proof} (1) follows from Proposition 3.8 of \cite{CoDiOl3}, and (2) follows from the definitions
and from \eqref{eq:pzp}.
\end{proof}

\begin{remark} As mentioned in \cite{Coh-Mor}, the congruence (1), or
equivalently the integrality of $a_{\al}(\p)$ (when $A_{\al}(\p)<M(\p)$)
comes from a subtle although very classical computation involving higher 
ramification groups; see Proposition 3.6 of \cite{CoDiOl3} along with 
Chapter 4 of \cite{serre_lf}.
\end{remark}

We can now quote the crucial result from \cite{CoDiOl3} which gives the conductor of the extension $N/K$:

\begin{theorem}\cite[Theorem 3.15]{CoDiOl3}\label{thmfnk} Assume that $(\a_0, \dots, \a_{\ell - 2})$ are as in 
Proposition \ref{propbij2}, so that 
$\prod_{0\le i\le\ell-2}\a_i=\a_{\al}\Z_{K_z}$
with $\a_{\al}$ an ideal of $K$ stable by $\tau_2$ (resp., by $\tau$ in
the special case), and sometimes coming from $k$ (see Proposition
\ref{propsplit2}). Then the conductor of the associated field extension $N/K$ is given as follows:
$$f(N/K)=\ell\a_{\al}\dfrac{\prod_{\p\mid\ell}\p^{\lceil e(\p)/(\ell-1)\rceil}}
{\prod_{\p\mid\ell\;,\ \p\nmid\a_{\al}}\p^{\lceil a_{\al}(\p)\rceil}}\;.$$
\end{theorem}

\begin{remark}\label{rk:corappar}
One can now draw 
additional conclusions about the $a_{\alpha}(\p)$. For example, suppose that
$p$ is a prime ideal $k$ above $\ell$
with $p\Z_K=\p^2$, $\p\nmid\a_{\al}$ and $a_{\al}(\p)<m(\p)$.
Then $v_{\p}(f(N/K)/\ell)\equiv0\pmod2$, as $f(N/K)=f(L)\Z_K$ for an ideal $f(L)$ of $k$,
and it follows from the theorem and Proposition \ref{propap} that
\begin{equation}\label{eq:corappar}
a_{\al}(\p)\equiv\lceil e(\p)/(\ell-1)\rceil\pmod{2}.
\end{equation}
\end{remark}

\begin{definition}\label{defha} 
Let $a$ equal either $m(\p)$, or an integer with
$0\le a < m(\p) - 1$, and define
$$h(0,a,\p)=\begin{cases}0&\text{\quad if $(\ell-1)\nmid e(\p)$ \ or \
$a = m(\p)$\;,}\\
1&\text{\quad if $(\ell-1)\mid e(\p)$ \ and \ $a < m(\p)$\;;}\end{cases}$$
$$h(1,a,\p)=\begin{cases}1&\text{\quad if $(\ell-1)\nmid e(\p)$\;,}\\
2&\text{\quad if $(\ell-1)\mid e(\p)$\;.}\end{cases}$$
\end{definition}

\begin{remark}\label{rem:ell3}
Note that if $\ell>2[k:\Q]+1$, for instance when
$k=\Q$ and $\ell\ge5$, we have $e(\p)<\ell-1$ so $(\ell-1)\nmid e(\p)$.
Thus in this case we simply have $h(\eps,a,\p)=\eps$, independently of
$a$ and $\p$. We will also see in Remark \ref{remre} that a number of
other formulas simplify.
\end{remark}

\smallskip

\begin{lemma}\label{lembdef} Let $\p$ be a prime ideal of $K$ above $\ell$ and denote by 
$C_n$ the congruence $x^\ell/\al\equiv1\pmod{^*\p^n}$ in $K_z$. Then 
$a_{\al}(\p)$ is equal to the unique value of $a$ as in the previous definition such that
$C_n$
is soluble for $n=a+h(0,a,\p)$ and not soluble for $n=a+h(1,a,\p)$,
where this last condition is ignored if $a+h(1,a,\p)>m(\p)$.
\end{lemma}

\begin{proof} 
By Lemma \ref{lemsol} the solubility of $E_n$ is equivalent
to that of $C_{n/e(\p_z/\p)}$. If 
$a=a_{\al}(\p)=m(\p)$, then $E_n$ is soluble for
$n=\ell e(\p_z/\ell)/(\ell-1)$, which is equivalent to 
$C_{m(\p)}=C_a$ as desired.

If $a=a_{\al}(\p)<m(\p)$, we have
$A_{\al}(\p)=ae(\p_z/\p)+r(e(\p))/(\ell-1,e(\p))$, and 
Proposition \ref{propap} (1) implies that the solubility
of $E_n$ for $n=A_{\al}(\p)$ is equivalent to that of $E_{n'}$ 
when $A_{\al}(\p)-(\ell-1)/(\ell-1,e(\p))<n'\le A_{\al}(\p)$.
If $(\ell-1)\nmid e(\p)$ we have $r(e(\p))<\ell-1$ and choose
$n'=ae(\p_z/\p)$, while if $(\ell-1)\mid e(\p)$ we choose 
$n'=n=ae(\p_z/\p)+1$. Thus the solubility of $E_{A_{\al}(\p)}$ and 
$E_{n'}$ is equivalent to that of $C_{n''}$, where
$n''=n'/e(\p_z/\p)=a+h(0,a,\p)$ by definition of $h(0,a,\p)$. 
(Recall that $e(\p_z/\p) = 1$ when $(\ell - 1)\mid e(\p)$.)

Furthermore, since
$E_n$ is not soluble for $n=A_{\al}(\p)+1$, we also have that
$E_{n'}$ is not soluble, where $n' = n$ if $(\ell - 1) \mid e(\p)$
and $n' = a(\p_z/\p) + (\ell-1)/(\ell-1,e(\p)) \geq n'$ otherwise.
The solubility of $E_{n'}$ is equivalent to that of $C_{n''}$ where
 $n''=n'/e(\p_z/\p)=a+h(1,a,\p)$, as desired.
 
 Finally, we conclude by checking that the conditions are mutually exclusive.

\end{proof}

\section{The Dirichlet Series}\label{sec:dir_series}

Since $f(N/K)=f(L)\Z_K$ for some ideal $f(L)$ of $k$, we have
$\N_{K/\Q}(f(N/K))=\N_{k/\Q}(f(L))^2$.
To
emphasize the fact that we are mainly interested in the norm from $k/\Q$, we set
the following definition (norms from extensions other than $k/\Q$ will always indicate the field extension explicitly):
\begin{definition}\label{def41} If $\a$ is an ideal of $k$, we set 
$\N(\a)=\N_{k/\Q}(\a)$, while if $\a$ is an ideal of $K$, we set 
$$\N(\a)=\N_{K/\Q}(\a)^{1/2}\;.$$
In particular, for each ideal $\a$ of $k$ we have $\N(\a) = \N(\a \Z_K)$.
\end{definition}

Recall that we set
$$\Phi_{\ell}(K,s)=\dfrac{1}{\ell-1}+\sum_{L\in\calFL(K)}\dfrac{1}{\N(f(L))^s}\;,$$
with $f(N/K)=f(L)\Z_K$ is given by Theorem \ref{thmfnk}. By Proposition 
\ref{propbij1}, we have
$$(\ell-1)\Phi_{\ell}(K,s)=\sum_{\ov{\al}\in(K_z^*/{K_z^*}^\ell)[T]}\dfrac{1}{\N(f(L))^s}\;,$$
where $L=K_z(\root\ell\of\al)^G$ (including $\ov{\al}=1$ corresponding to $L=K_z^G=k$ with $f(L)=\Z_k$
and $\N(f(L))=1$), so by Proposition \ref{propbij2}, we have
$$(\ell-1)\Phi_{\ell}(K,s)=\sum_{(\a_0,\dots,\a_{\ell-2})\in J}\sum_{\ov{u}\in S_{\ell}(K_z)[T]}\dfrac{1}{\N(f(L))^s}\;,$$
\noindent
where $J$ is the set of $(\ell-1)$uples of ideals satisfying condition (a) of
Proposition \ref{propbij2}, and $f(L)$ is the conductor of the extension 
corresponding to $(\a_0,\dots,\a_{\ell-2},\ov{u})$. Thus, replacing $f(L)$ by
the formula given by Theorem \ref{thmfnk}, recalling that $\prod_{\p | \ell} \N(\p)^{e(\p)}
= \ell^{[k : \Q]}$, and writing
$$e(\p)=(\lceil e(\p)/(\ell-1)\rceil-1)(\ell-1)+r(e(\p))\;,$$
we obtain
\begin{equation}\label{eq:phil}
(\ell-1)\Phi_{\ell}(K,s)=
\ell^{-\frac{\ell}{\ell-1}[k:\Q]s}
\prod_{\p\mid\ell}\N(\p)^{-\frac{\ell-1-r(e(\p))}{\ell-1}s}
\sum_{(\a_0,\dots,\a_{\ell-2})\in J}\dfrac{S_{\al}(s)}{\N(\a_{\al})^s}\;,
\end{equation}
where
$$S_{\al}(s)=\sum_{\ov{u}\in S_{\ell}(K_z)[T]}\prod_{\substack{\p\mid\ell\\\p\nmid\a_{\al}}}\N(\p)^{\lceil a_{\al u}(\p)\rceil s}\;,$$
and where $\al$ is any element of $K_z^*$ such that
$\ov{\al}\in(K_z^*/{K_z^*}^\ell)[T]$ and $\q_0^\ell\prod_{0\le i\le\ell-2}\a_i^{g^i}=\al\Z_{K_z}$
for some ideal $\q_0$.

\begin{definition}\label{def:fal} For $\al\in K_z^*$ and an ideal $\b$ of $K_z$,
we introduce the function
$$f_{\al}(\b)=|\{\ov{u}\in S_{\ell}(K_z)[T],\ x^{\ell}/(\al u)\equiv1\pmod{^*\b}\text{\quad soluble in $K_z$}\}|\;,$$
with the convention that $f_{\al}(\b)=0$ if $\b\nmid(1-\ze)^\ell\Z_{K_z}$.
\end{definition}

Let $p_i$ for $1\le i\le n = n(\alpha)$ be the prime ideals of $k$ above $\ell$ and
not dividing $\a_{\al}$, and for each $i$ let $a_i$ be such 
that either $a_i=m(\p_i)$, or 
$0\le a_i \leq m(\p_i)-\frac{(\ell-1)+r(e(\p_i))}{\ell-1} = \lceil m(\p_i) \rceil - 2$ with
$a_i\in\Z$, where as usual $\p_i$ is an ideal of $K$ above $p_i$, and let $A$ be 
the set of such $(a_1,\dotsc,a_n)$. Noting that thanks to the convention
of Definition \ref{def41} we have 
$\prod_{\p_i\mid p_i}\N(\p_i)=\N(p_i)^{1/e(\p_i/p_i)}$, we thus have

\begin{equation}\label{eq:sal_prelim}
S_{\al}(s)=\sum_{(a_1,\dotsc,a_n)\in A}\prod_{1\le i\le n}\N(p_i)^{\lceil a_i\rceil s/e(\p_i/p_i)}
\sum_{\substack{\ov{u}\in S_{\ell}(K_z)[T]\\\forall{i},\ a_{\al u}(\p_i)=a_i}}1\;.
\end{equation}

By Lemma \ref{lembdef}, we have $a_{\alpha u}(\p_i) \geq a_i$ if and only if $\ov{u}$
is counted by $f_{\al}(\p_i^{b_i})$, where $b_i = a_i + h(0, a, \p_i)$, and we rewrite
$\p_i^{b_i} = p_i^{b_i/e(\p_i/p_i)}$. 
Let $B(\alpha)$ be the
set of $n$-uples $(b_1,\dotsc,b_n)$ with $0\le b_i\le m(\p_i)$,
$b_i\in \Z\cup \{m(\p_i)\}$.
By inclusion-exclusion we obtain the following:

\begin{lemma}\label{lempb} We have
\begin{equation}\label{eqn_sal}
S_{\al}(s)=\sum_{(b_1,\dots,b_n)\in B(\alpha)} f_{\al}\left(\prod_{1\le i\le n}p_i^{b_i/e(\p_i/p_i)}\right) \prod_{1\le i\le n} 
\left(\N(p_i)^{\lceil b_i\rceil s/e(\p_i/p_i)} Q(p_i^{b_i/e(\p_i/p_i)},s)\right),
\end{equation}
where $Q(p^{b/e(\p/p)},s)$ is defined as follows. Let as usual $\p$ be an 
ideal of $K$ above $p$ and define 
$q=\N(p)^{1/e(\p/p)}$. Then if $b=m(\p)$ or $0\le b<m(\p)$ with $b\in\Z$:
\begin{enumerate}
\item If $(\ell-1)\nmid e(\p)$ we set
$$Q(p^{b/e(\p/p)},s)=\begin{cases}
1 & \text{\quad if\quad}b=0\;,\\
1-1/q^s &\text{\quad if\quad}1\le b\le \lceil m(\p)\rceil-2\;,\\
-1/q^s &\text{\quad if\quad}b=\lceil m(\p)\rceil-1\;,\\
1 & \text{\quad if\quad}b=m(\p)\;.\end{cases}$$
\item If $(\ell-1)\mid e(\p)$ we set
$$Q(p^{b/e(\p/p)},s)=\begin{cases}
0 & \text{\quad if\quad}b=0\;,\\
1/q^s & \text{\quad if\quad}b=1 \;,\\
1/q^s-1/q^{2s} & \text{\quad if\quad}2\le b\le m(\p)-1\;,\\
1-1/q^{2s} & \text{\quad if\quad}b=m(\p)\;.\end{cases}$$
\end{enumerate}
\end{lemma}

\medskip
\begin{remark}
There are conditions on the $a_i$, e.g.
\eqref{eq:corappar}, such that the inner sum in \eqref{eq:sal_prelim}
vanishes for impossible choices
of the $a_i$. One can use this to prove alternate versions of Lemma \ref{lempb}
that are nonobviously equivalent. In particular, if $(\ell - 1) \mid e(\p)$ then one can restrict
to $b_i \in 2\Z \cup \{ m(\p_i) \}$ with suitably modified $Q(p^{b/e(\p/p)}, s)$.
\end{remark}
\smallskip

\begin{definition}\label{defB}
\begin{enumerate}\item We let $\mathcal{B}$ be the set of formal products of 
the form \\$\b=\prod_{p_i\mid \ell}p_i^{b_i/e(\p_i/p_i)}$, where the $b_i$ are
such that $0\le b_i\le m(\p_i)$ and $b_i\in\Z\cup \{m(\p_i)\}$.
\item We will consider any $\b\in\mathcal{B}$ as an ideal of $K$, where by
abuse of language we accept to have fractional powers of prime ideals of $K$, 
and we will set $\b_z=\b\Z_{K_z}$, which is a true ideal of $K_z$ stable by
$\tau$, and also by $\tau_2$ in the general case.
\item If $\b\in\mathcal{B}$ as above, we set
$$\NC(\b)=\prod_{p_i\mid\b}\N(p_i)^{\lceil b_i\rceil/e(\p_i/p_i)}\text{\quad and\quad}P(\b,s)=
\prod_{p_i\mid\b}\widetilde{Q}(p_i^{b_i/e(\p_i/p_i)},s)\;.$$
where $\widetilde{Q}(p^{b/e(\p/p)},s) := Q(p^{b/e(\p/p)},s)$
except in the case $(\ell - 1) \mid e(\p)$ and $b = 0$, where we set
$\widetilde{Q}(p^{b/e(\p/p)},s) = 1$.
\end{enumerate}
\end{definition}

We thus obtain
\begin{equation}\label{eqn:before_reb}
\sum_{(\a_0,\dots,\a_{\ell-2})\in J}\dfrac{S_{\al}(s)}{\N(\a_{\al})^s}=
\sum_{\b\in\mathcal{B}} \NC (\b)^s P(\b,s)\sum_{\substack{(\a_0,\dots,\a_{\ell-2})\in J\\  (\a_{\al},\b)=1\\ \p\nmid \b\text{ and }(\ell-1)\mid e(\p)\Rightarrow \p\mid \a_{\al} }}\dfrac{f_{\al}(\b)}{\N(\a_{\al})^s}\;.
\end{equation}
The case $\p \nmid \b$, $(\ell - 1) \mid e(\p)$, and $\p \nmid \a_{\alpha}$ is precisely that for which
$Q(p^{b/e(\p/p)},s) = 0$ and $\widetilde{Q}(p^{b/e(\p/p)},s) = 1$. By excluding this case
we may substitute $\widetilde{Q}$ for $Q$ with $\widetilde{Q}(p^0, s) = 1$.

\begin{definition}\label{defd3}\begin{enumerate}\item For $\b$ as above we 
define
$$\gothr^e(\b)=\prod_{\substack{\p\mid \ell\Z_K,\ \p\nmid \b\\ (\ell-1)\mid e(\p)}}\p\;.$$
\item We set
$\gd_\ell=\prod_{p\in\calD_\ell}p$ (see Definition \ref{defcald}).
\end{enumerate}
\end{definition}

\begin{remark}\label{remre} 

Since $e(\p)=e(\p/p)e(p)\le 2[k:\Q]$,
we note that if $\ell>2[k:\Q]+1$ then $\gothr^e(\b)$ is always trivial, so that
specializing to the case $k=\Q$ and $\ell\ge5$ now would avoid some complications.

\end{remark}

\begin{lemma}\label{lem:reb} For each $\a_{\al}$ appearing in the inner sum of \eqref{eqn:before_reb} we have
\begin{equation}\label{eqn:reb_eq}
(\a_{\al},\ell\Z_K)=\gothr^e(\b) = \prod_{\substack{p \in D_{\ell} \\ (p, \b) = 1}} \prod_{\p | p} \p
\;,
\end{equation}
so that $\gothr^e(\b)\mid\gd_\ell$.

Additionally, in the special case with $\ell \equiv 1 \pmod 4$ we have
$\gothr^e(\b) = \prod_{\substack{p \in D_{\ell} \\ (p, \b) = 1}} p.
$
\end{lemma}

\begin{proof} If $\p\nmid\b$ and $(\ell-1)\mid e(\p)$
then clearly $\p\mid\a_{\al}$. Conversely, let $\p\mid\a_{\al}$ be above $\ell$. Since 
$(\a_{\al},\b)=1$ we know that $\p\nmid\b$. If we had
$(\ell-1)\nmid e(\p)$, Proposition \ref{proppzp} would imply that 
$e(\p_z/\p)>1$, 
contradicting Corollary \ref{corsplit1}. This proves the first equality of \eqref{eqn:reb_eq}, and
the rest follows similarly. 
\end{proof}

Thus we obtain
\begin{equation}\label{eq:jsum}
\sum_{(\a_0,\dots,\a_{\ell-2})\in J}\dfrac{S_{\al}(s)}{\N(\a_{\al})^s}=
\sum_{\substack{\b\in\mathcal{B}\\\gothr^e(\b)\mid\gd_\ell}} \NC (\b)^s
  P(\b,s)\sum_{\substack{(\a_0,\dots,\a_{\ell-2})\in J
 \\ (\a_{\al},\ell\Z_K)=\gothr^e(\b)}}
\dfrac{f_{\al}(\b)}{\N(\a_{\al})^s}.
\end{equation}

To compute $f_{\al}(\b)$ we set the following definition:

\begin{definition}\label{defsb} For any ideal $\b\in\calB$, and for any subset $T$
of $\F_\ell[G]$, we set
$$S_{\b_z}(K_z)[T]=\{\ov{u}\in S_{\ell}(K_z)[T],\ x^\ell/u\equiv1\pmod{^*\b_z}\text{ soluble}\}\;,$$
where $u$ is any lift of $\ov{u}$ coprime to $\b_z$, and the congruence is
in $K_z$.\end{definition}

\begin{lemma}\label{lem:sizefal} Let $(\a_0,\dots,\a_{\ell-2})$ satisfy condition (a) of
Proposition \ref{propbij2}, suppose that $\alpha$ satisfies the condition described before
Definition \ref{def:fal}, 
and recall that we set $\a=\prod_i\a_i^{g^i}$. We 
have
$$f_{\al}(\b)=\begin{cases}|S_{\b_z}(K_z)[T]|&\text{\quad if $\ov{\a}\in\Cl_{\b_z}(K_z)^\ell$\;,}\\
0&\text{\quad otherwise.}\end{cases}$$
\end{lemma}

\begin{proof} 
The lemma and its proof are a direct generalization of Lemma 5.3 of \cite{Coh-Mor}, and we omit the details.
\end{proof}

\section{Computation of $|S_{\b_z}(K_z)[T]|$}\label{sec:comp_gb}

In this section we compute the size of the group $S_{\b_z}(K_z)[T]$ appearing in Lemma
\ref{lem:sizefal}, as well as several related quantities.

\begin{lemma}\label{lemfal3} Set $Z_{\b_z}=(\Z_{K_z}/\b_z)^*$. Then
$$|S_{\b_z}(K_z)[T]|=\dfrac{|(U(K_z)/U(K_z)^\ell)[T]| 
|(\Cl_{\b_z}(K_z)/\Cl_{\b_z}(K_z)^\ell)[T]|}{|(Z_{\b_z}/Z_{\b_z}^\ell)[T]|}\;,$$
and in particular
$$|S_\ell(K_z)[T]|=|(U(K_z)/U(K_z)^\ell)[T]||(\Cl(K_z)/\Cl(K_z)^\ell)[T]|\;.$$
\end{lemma}

\begin{proof}
This is a minor variant of Corollary 2.13 of \cite{CoDiOl3}, proved
in the same way.
\end{proof}

The quantity $|(U(K_z)/U(K_z)^\ell)[T]|$ is given by the following lemma.

\begin{lemma}\label{lemu5} Assume that $\ell > 3$, the case $\ell = 3$ being treated in \cite[Lemma 5.4]{Coh-Mor}.
For any number field $M$, write
$\rk_\ell(U(M)):=\dim_{\F_\ell}(U(M)/U(M)^\ell)$, 
and denote by $r_1(M)$ and $r_2(M)$ the number
of real and pairs of complex embeddings of $M$.
\begin{enumerate}\item For any number field $M$ we have
$$\rk_\ell(U(M))=\begin{cases} 
r_1(M)+r_2(M)-1 & \text{ if $\ze\notin M$,}\\
r_1(M)+r_2(M) & \text{ if $\ze\in M$.}\end{cases}$$
\item We have $|(U(K_z)/U(K_z)^\ell)[T]|=\ell^{RU(K)}$, where
$$RU(K):=\begin{cases}
r_2(K)-r_2(k)&\text{\quad in the general case,}\\
r_1(k)+r_2(k)&\text{\quad in the special case with $\ell\equiv3\pmod4$\;,}\\
r_2(k)&\text{\quad in the special case with $\ell\equiv1\pmod4$\;.}
\end{cases}$$
\item In particular, if $k=\Q$ we have $RU(K)=r_2(K)$ in all cases.
\end{enumerate}
\end{lemma}

\begin{proof} (1) is Dirichlet's theorem, and (3) is a consequence of (2). 
To prove (2) in the general case, where $T=\{\tau_2+1,\tau-g\}$,
we apply the exact sequence
\begin{equation}\label{eq:es_52}
1\LR\dfrac{U(k_z)}{U(k_z)^\ell}[\tau-g]\LR\dfrac{U(K_z)}{U(K_z)^\ell}[\tau-g]\LR\dfrac{U(K_z)}{U(K_z)^\ell}[\tau_2+1,\tau-g]\LR1\;,
\end{equation}
where the last nontrivial map sends
$\eps$ to $\tau_2(\eps)/\eps$. Surjectivity follows from
Lemma \ref{lem24}, and
$(\tau_2+1)(\tau_2-1)=0$ implies that the two nontrivial maps compose to zero.
Finally, suppose $\eps\in U(K_z)$
satisfies $\tau_2(\eps)=\eps\eta^\ell$ for some $\eta\in K_z$. Applying
$\tau_2$ to both sides we see that $\eta\tau_2(\eta)=\ze^a$ for some $a$,
and replacing $\eta$ with $\eta_1 = \eta \ze^b$ with $a+2b\equiv0\pmod\ell$,
we obtain
$\eta_1\tau_2(\eta_1)=1$ and 
$\tau_2(\eps)=\eps\eta_1^\ell$. By Hilbert 90 
there exists $\eta_2$ with $\eta_1=\eta_2/\tau_2(\eta_2)$, so that $\eps_1=\eps\eta_2^\ell$
satisfies $\tau_2(\eps_1)=\eps_1$, in other words $\eps_1\in k_z$, proving exactness of
\eqref{eq:es_52}.

By a nontrivial theorem of Herbrand (see Theorem
2.3 of \cite{CoDiOl3}), we have
$|(U(K_z)/U(K_z)^\ell)[\tau-g]|=\ell^{r_2(K)+1}$ and
$|(U(k_z)/U(k_z)^\ell)[\tau-g]|=\ell^{r_2(k)+1}$, establishing (2) in the general case.

\smallskip

In the special case, with
$T=\{\tau+g\}=\{\tau-g^{(\ell+1)/2}\}$, (2) follows directly from Herbrand's theorem
applied to the extension $k_z / k = K_z /k$, for which $\tau$ generates the Galois group.
\end{proof}

Note that for $\ell = 3$ the same is true except that in the special case we 
have $RU(K) = r_1(k) + r_2(k) - 1$. This follows from the shape
of \cite[Theorem 2.3]{CoDiOl3}, or may be easily verified from 
\cite[Lemma 5.4]{Coh-Mor}.

%We now need to compute $|(Z_{\b_z}/Z_{\b_z}^\ell)[T]|$.

\begin{lemma}\label{lemzz} Let $\b\in\calB$ 
satisfy $\b_z\mid(1-\ze)^\ell$, and define
$\c_z=\prod_{\substack{\p_z\subset K_z\\\p_z\mid\b_z}}\p_z^{\lceil v_{\p_z}(\b_z)/\ell\rceil}$. We have
$$|(Z_{\b_z}/Z_{\b_z}^\ell)[T]|=|(\c_z/\b_z)[T]|\;,$$
the latter being considered as an additive group.
\end{lemma}

\begin{proof} See Proposition 2.6 and Theorem 2.7 of \cite{CoDiOl3}, or Lemma 1.5.6 of \cite{Mor}.
\end{proof}

\begin{theorem}\label{thmcb} 

We have in the general case
\begin{equation}
|(\c_z/\b_z)[\tau-g^j]|=\prod_{\p\mid\b_z}\N_{K/\Q}(\p)^{x_j(\p)}\;,
\end{equation}
where 
\begin{equation}\label{eq:xp}
x_j(\p) = 
\Big\lceil v_\p(\b) - \frac{ j e(\p) }{\ell - 1} \Big\rceil - 
\Big\lceil \frac{ \lceil e(\p_z/\p) v_\p(\b) / \ell \rceil (\ell - 1, e(\p)) - je(\p)}{\ell - 1} \Big\rceil.
\end{equation}

In the special case, \eqref{eq:xp} holds with $\p$ and $K$ replaced throughout by
$p$ and $k$ respectively.

Finally, in the general case, then \eqref{eq:xp} is also true with respect to $k_z/k$. In this case one must replace
$\p$, $K$, $\b_z$, and $\c_z$ respectively by $p$, $k$, $b_k := \c_z \cap k_z$, and
\begin{equation}\label{eq:ck}
\c_k := \c_z \cap \Z_{k_z} = 
\prod_{\substack{\p_k\subset k_z\\\p_k\mid\b_k}}\p_k^{\lceil v_{\p_k}(\b_k)/\ell\rceil}.
\end{equation}
\end{theorem}

\begin{proof}
This is the result at the bottom of \cite[p. 177]{CoDiOl3}, applied to $K_z/K$, $K_z/k$, and $k_z/k$ respectively.
As in \cite[Theorem 2.7]{CoDiOl3} the result may be simplified if 
$v_\p(\b)$ is either an integer or equal to $\frac{\ell e(\p)}{\ell - 1}$, and in particular always in the general
case with respect to $K_z/K$, but in other cases $v_p$ may be a half integer.

Finally, the equality in \eqref{eq:ck} is readily verified.
\end{proof}

Recall from \cite[Theorem 2.1]{CoDiOl3} and \eqref{eq:pzp} that $e(\p_z/\p) = \frac{\ell - 1}{(\ell - 1, e(\p))}$
and $e(\p_k/p) = \frac{\ell - 1}{(\ell - 1, e(p))}$. In the
special case, this theorem together with Lemma \ref{lemzz} gives the
cardinality of $(Z_{\b_z}/Z_{\b_z}^\ell)[T]$ by choosing $j=(\ell+1)/2$. In the 
general case we require the following additional lemma:

\begin{lemma}\label{lem:27_es}
Assume that we are in the general case and set 
$\c_k=\c_z\cap k_z$ and $\b_k=\b_z\cap k_z$. We have
$$|(Z_{\b_z}/Z_{\b_z}^\ell)[T]|=|(\c_z/\b_z)[\tau-g]|/|(\c_k/\b_k)[\tau-g]|\;,$$
where the two terms on the right-hand side are given by Theorem \ref{thmcb}.
\end{lemma}

\begin{proof} We have an exact sequence of
$\F_\ell[G]$-modules
$$1\LR\dfrac{\c_z}{\b_z}[\tau_2-1][\tau-g]\LR\dfrac{\c_z}{\b_z}[\tau-g]\LR\dfrac{\c_z}{\b_z}[T]\LR1\;,$$
the last map sending $x$ to $x - \tau_2(x)$. It therefore suffices to argue that
$(\c_z/\b_z)[\tau_2-1]=(\c_z\cap k_z)/(\b_z\cap k_z)$: if $x\in\c_z$
satisfies $\tau_2(x)=x+y$ for some $y\in\b_z$, then applying $\tau_2$ we see that
$\tau_2(y)=-y$, hence $\tau_2(x+y/2)=x+y/2$. Moreover
$x+y/2\equiv x\pmod{\b_z}$, because $2$ is invertible modulo $\ell$ hence
modulo $\b$.
\end{proof}

\begin{definition} We set $G_{\b}=(\Cl_{\b_z}(K_z)/\Cl_{\b_z}(K_z)^\ell)[T]$.
\end{definition}

We conclude with one additional lemma which will be needed in the next section.

\begin{lemma}\label{lemt1t2} In the general case set 
$u=\i(\tau_2+1)\i(\tau-g)$ and in the special case set $u=\i(\tau+g)$.
\begin{enumerate}
\item The map $I\mapsto u(I)$ induces a surjective
map from $\Cl_{\b_z}(K_z)/\Cl_{\b_z}(K_z)^\ell$ to $G_{\b}$, of which a section
is the natural inclusion from $G_{\b}$ to $\Cl_{\b_z}(K_z)/\Cl_{\b_z}(K_z)^\ell$.
\item Any character $\chi\in\widehat{G_{\b}}$ can be naturally
extended to a character of $\Cl_{\b_z}(K_z)/\Cl_{\b_z}(K_z)^\ell$ by setting
$\chi(\ov{I})=\chi(\ov{u(I)})$, which we again denote by $\chi$ by abuse of 
notation.
\item Let as usual $\a=\prod_{0\le i\le\ell-2}\a_i^{g^i}$ with the $\a_i$
satifying condition (a) of Proposition \ref{propbij2}.
\begin{itemize}\item In the general case and in the special case when
$\ell\equiv1\pmod4$, we have $\chi(\ov{\a})=\chi(\ov{\a_0})^{-1}$;
\item In the special case when $\ell \equiv 3\pmod 4$, we have
$\chi(\ov{\a})=\chi(\ov{\a_0\a_1^g})^{(\ell-1)/2}$,
\end{itemize}
where $\chi$ on the right-hand side is defined in (2).
\end{enumerate}
\end{lemma}

\begin{proof} (1) and (2) are immediate from Lemma \ref{lem24}. 
For (3), assume that we are in the special case. Using Lemma \ref{lemconda1} we have $\a_{2i}=\tau^{-2i}(\a_0)$,
$\a_{2i+1}=\tau^{-2i}(\a_1)$, and 
$\chi(\ov{\tau^2(I)})=\chi(\ov{I})^{g^2}$, so that
$$\chi(\ov{\a})=\prod_{0\le i<(\ell-1)/2}\chi(\ov{\tau^{-2i}(\a_0\a_1^g)})^{g^{2i}}=\prod_{0\le i<(\ell-1)/2}\chi(\ov{\a_0\a_1^g})=\chi(\ov{\a_0\a_1^g})^{(\ell-1)/2}\;.$$
If in addition $\ell\equiv1\pmod4$ we have $\a_1=\tau^{(\ell-3)/2}(\a_0)$
and $\chi(\ov{\tau(I)})=\chi(\ov{I^{-g}})$, giving
$\chi(\ov{\a_1})=\chi(\ov{\a_0})^{(-g)^{(\ell-3)/2}}=\chi(\ov{\a_0})^{-g^{(\ell-3)/2}}$
and $\chi(\ov{\a_1^g})=\chi(\ov{\a_0})$, so 
$\chi(\ov{\a_0\a_1^g})^{(\ell-1)/2}=\chi(\ov{\a_0})^{\ell-1}=\chi(\ov{\a_0})^{-1}$.

The general case of (3) is proved similarly, with $\a_i = \tau^{-i}(\a_0)$.

\end{proof}

\section{Semi-Final Form of the Dirichlet Series}\label{sec:semifinal}

We can now put everything together, and obtain a complete analogue of the 
main theorem of \cite{Coh-Mor}:

\begin{theorem}\label{thm:main} 
%Recall that for any (true or formal) ideal $\b$ of $K$ as above we set
%$G_{\b}=(\Cl_{\b_z}(K_z)/\Cl_{\b_z}(K_z)^\ell)[T]$. 
We have
\begin{align*}\Phi_{\ell}(K,s)=\dfrac{\ell^{RU(K)}}{(\ell-1)\ell^{\frac{\ell}{\ell-1}[k:\Q]s}}
&\prod_{\p\mid\ell}\N(\p)^{-\frac{(\ell-1-r(e(\p))}{\ell-1}s}\;\cdot\\
&\phantom{=}\cdot\sum_{\substack{\b\in\mathcal{B}\\\gothr^e(\b)\mid\gd_\ell}}
\left(\dfrac{\NC(\b)}{\N(\gothr^e(\b))}\right)^s\dfrac{P(\b,s)}{|(Z_{\b_z}/Z_{\b_z}^\ell)[T]|}\sum_{\chi\in\widehat{G_{\b}}}F(\b,\chi,s) \;,\end{align*}
where
$$F(\b,\chi,s)=\prod_{\substack{p\mid\gothr^e(\b)\\p\in\mathcal{D}_\ell'(\chi)}}(\ell-1)\prod_{\substack{p\mid\gothr^e(\b)\\p\in\mathcal{D}_\ell\setminus{\mathcal{D}_\ell}'(\chi)}}(-1)\prod_{p\in\mathcal{D}'(\chi)} \left(1+\dfrac{\ell-1}{\N(p)^{s}}\right)\prod_{p\in\mathcal{D}\setminus\mathcal{D}'(\chi)} \left(1-\dfrac{1}{\N(p)^{s}}\right)\;,$$
and $\mathcal{D}'(\chi)$ (resp. $\mathcal{D}_\ell'(\chi)$) is the set of
$p\in\mathcal{D}$ (resp. $\mathcal{D}_\ell$) such that 
$\chi(\p_z)=1$, where $\p_z$ is any prime ideal of $K_z$ above $p$.
\end{theorem}

\begin{proof} 
We begin with the formula for $\Phi_{\ell}(K, s)$ given by 
\eqref{eq:phil} and \eqref{eq:jsum}. By Remark \ref{remclk} 
we have $\ov{\a}\in(\Cl_{\b_z}(K_z)/\Cl_{\b_z}(K_z)^\ell)[T]$ with
$\mfa = \prod_{0 \leq i \leq \ell - 2} \mfa_i^{g^i}$.
Thus $\ov{\a}\in \Cl_{\b_z}(K_z)^\ell$ if and only if 
$\chi(\ov{\a})=1$ for all characters $\chi\in\widehat{G_{\b}}$. The number of
such characters being equal to $|G_{\b}|$, by orthogonality of characters and
Lemmas \ref{lem:sizefal}, \ref{lemfal3}, and \ref{lemu5} we obtain
$$\Phi_\ell(K,s)=\dfrac{\ell^{RU(K)}}
{(\ell-1)\ell^{\frac{\ell}{\ell-1}[k:\Q]s}}\prod_{\p\mid \ell}
\N(\p)^{-\frac{\ell-1-r(e(\p))}{\ell-1}s}\sum_{\substack{\b\in\mathcal{B}\\\gothr^e(\b)\mid\gd_\ell}}\dfrac{\NC(\b)^sP(\b,s)}
{|(Z_{\b_z}/Z_{\b_z}^\ell)[T]|}\sum_{\chi\in\widehat{G_{\b}}}H(\b,\chi,s)\;,$$
with
$$H(\b,\chi,s)=\sum_{\substack{(\a_0,\cdots,\a_{\ell-2})\in J'\\(\a_{\al},\ell\Z_K)=\gothr^e(\b)}}\dfrac{\chi(\ov{\a})}{\N(\a_{\al})^s}\;,$$
where $\mfa_{\alpha}$ was defined in Proposition \ref{propsplit2}, and
$J'$ is the set of $(\ell-1)$uples of coprime squarefree ideals of 
$K_z$, satisfying condition (a) of Proposition \ref{propbij2}, but now
without the condition that the ideal class of $\mfa$ belongs to $\Cl(K_z)^{\ell}$, so satisfying the condition of
Lemma \ref{lemconda1}.

{\bf Assume first that we are in the general case.} By Lemma \ref{lemconda1} we can
replace the sum over $J'$ by a sum over ideals $\a_0$ of $K_z$. The 
conditions and quantities linked to $\a_0$ are then as follows:

\begin{enumerate}[(a)]
\item The ideal $\a_0$ is a squarefree ideal of $K_z$ such that
$\tau^{(\ell-1)/2}(\a_0)=\tau_2(\a_0)$.
\item The ideals $\a_0$ and $\tau^i(\a_0)$ are coprime for $(\ell-1)\nmid i$.
\item If $\p_z$ is a prime ideal of $K_z$ dividing $\a_0$, $\p$ the prime 
ideal of $K$ below $\p_z$, and $p$ the prime ideal of $k$ below $\p_z$ then by
Corollary \ref{corsplit1} we have $p\in \calD\cup\calD_\ell$. Conversely, if 
this is satisfied then the ideals $\a_i=\tau^{-i}(\a_0)$ must be
pairwise coprime since otherwise $\a_{\al}$ would be divisible by some 
$\p_z^2$Â¸ which is impossible since $\p$ is unramified in $K_z/K$.
\item We have $\N_{K_z/K}(\a_0)=\a_{\al}$.
\item By Lemma \ref{lemt1t2} we have $\chi(\ov{\a})=\chi^{-1}(\ov{\a_0})$.
\end{enumerate}

Thus if we denote temporarily by $J''$ the set of ideals $\a_0$ of $K_z$
satisfying the first three conditions above, we have
$$H(\b,\chi,s)=\sum_{\substack{\a_0\in J''\\(\N_{K_z/K}(\a_0),\ell\Z_K)=\gothr^e(\b)}}\dfrac{\chi^{-1}(\ov{\a_0})}{\N(\N_{K_z/K}(\a_0))^s}\;.$$

So that we can use multiplicativity, write $\a_0=\c\d$, where $\c$ is the $\ell$-part of
$\a_0$ and $\d$ is the prime to $\ell$ part (recall that $\a_0$ is squarefree).
The condition $(\N_{K_z/K}(\a_0),\ell\Z_K)=\gothr^e(\b)$ is thus equivalent
to $\N_{K_z/K}(\c)=\gothr^e(\b)$. Thus $H(\b,\chi,s)=S_cS_d$ with
$$S_c=\sum_{\substack{\c\in J''\\\N_{K_z/K}(\c)=\gothr^e(\b)}}\dfrac{\chi^{-1}(\ov{\c})}{\N(\N_{K_z/K}(\c))^s}\text{\quad and\quad}S_d=\sum_{\substack{\d\in J''\\(\N_{K_z/K}(\d),\ell\Z_K)=1}}\dfrac{\chi^{-1}(\ov{\d})}{\N(\N_{K_z/K}(\d))^s}\;.$$

Consider first the sum $S_d$. By multiplicativity we have
$S_d=\prod_{p\in\calD}S_{d,p}$ with
$$S_{d,p}=\sum_{\substack{\d\mid p\Z_{K_z}\\\tau^{(\ell-1)/2}(\d)=\tau_2(\d)}}\dfrac{\chi^{-1}(\ov{\d})}{\N(\N_{K_z/K}(\d))^s}\;.$$

As $p$ is not above $\ell$, it is unramified in
$K/k$ by Proposition \ref{propkp} and we consider the remaining two cases:

\begin{enumerate}\item Assume that $p\Z_K=\p$, i.e, that $p$ is inert in $K/k$.
Since $\p$ is totally split in $K_z/K$ we have 
$\p\Z_{K_z}=\prod_{0\le i\le\ell-2}\tau^i(\p_z)$ for some prime ideal $\p_z$
of $K_z$. Furthermore, since $\p_z/\p_k$ (with our usual notation) is split
we have $\tau_2(\p_z)\ne\p_z$, and since $\p$ is stable by $\tau_2$, 
$\tau_2(\p_z)$ is again above $\p$, so 
we deduce that $\tau_2(\p_z)=\tau^{(\ell - 1)/2}(\p_z)$.

Since $\d$ is squarefree and coprime to its $K_z/K$-conjugates,
we see that $\d=\Z_{K_z}$ or $\d=\tau^i(\p_z)$ for some $i$, with
$\N(\N_{K_z/K}(\d))$ equal to $1$ or $\N(p)$ respectively. In the latter case we have

\begin{equation}\label{eqn:sdp}
S_{d,p}=1+\sum_{0\le i\le\ell-2}\dfrac{\chi(\ov{\p_z})^{-g^i}}{\N(p)^s}=
1+\sum_{1\le j\le\ell-1}\dfrac{\chi(\ov{\p_z})^j}{\N(p)^s}\;,
\end{equation}
so that $S_{d,p}=1+(\ell-1)/\N(p)^s$ if $\chi(\ov{\p_z})=1$, and
$S_{d,p}=1-1/\N(p)^s$ otherwise.

\item If instead $p\Z_K=\p\tau_2(\p)$ is split in $K/k$, then similarly either
$\d = \Z_{K_z}$ or 
$\d=\tau^i(\p_z \tau^{\frac{\ell - 1}{2}} \tau_2(\p_z))$ for some $i$ and $\p_z$.
We have that
$\chi(\tau^{(\ell-1)/2}(\tau_2(\p_z)))=\chi^{-1}(\tau_2(\p_z))=\chi(\p_z),$
and hence obtain
the same result as above.

\end{enumerate}

Consider now the sum $S_c$. By multiplicativity, since $\b$ is stable by
$\tau_2$, and applying Lemma \ref{lem:reb} we have
$$S_c=\dfrac{1}{\N(\gothr^e(\b))^s}\sum_{\substack{\c\in J''\\\N_{K_z/K}(\c)=\gothr^e(\b)}}\chi^{-1}(\ov{\c})=\dfrac{1}{\N(\gothr^e(\b))^s}\prod_{\substack{p\in\calD_\ell\\(p,\b)=1}}S_{c,p}\;,$$
with
$$S_{c,p}=\sum_{\substack{\c\mid p\Z_{K_z}\\\tau^{(\ell-1)/2}(\c)=\tau_2(\c)\\
\N_{K_z/K}(\c)=\prod_{\p\mid p}\p}}\chi^{-1}(\ov{\c})\;.$$
Our analysis is essentially the same as before, except $p$ can now be ramified in $K/k$
and the possibility 
$\c = \Z_{K_z}$ is now excluded. In all cases we
obtain that 
$S_{c,p}=\ell-1$ if $\chi(\ov{\p_z})=1$ and
$-1$ otherwise.

Putting everything together proves the theorem in the general case.

\medskip

{\bf In the special case with $\ell\equiv1\pmod4$} the proof is similar; condition (a) is absent
and (d) becomes $\N_{K_z/k}(\mfa_0) = \N(\a_{\al})$.  Imitating the inert case of the previous argument, we obtain
the same results.

\medskip

{\bf In the special case with $\ell\equiv3\pmod4$,}
we replace the sum over $J'$ by a sum over pairs $(\a_0,\a_1)$ of ideals
of $K_z$ satisfying suitable conditions: 
\begin{itemize}
\item
In place of (a), $\a_0$ and $\a_1$ are fixed by $\tau^{(\ell - 1)/2}$.
\item
In place of (b), the ideals $\a_0$, $\a_1$,
$\tau^{2i}(\a_0)$, and $\tau^{2i}(\a_1)$ must all be coprime.
\item
In place of (d), we have $\N_{K_z/K}(\mfa_0 \mfa_1) = \N(\a_{\al})$.
\item
In place of (e), we have
$\chi(\ov{\a}) = \chi(\ov{\a_0 \a_1^g})^{(\ell - 1)/2}.$
\end{itemize}
We must again consider all splitting types in $K/k$, and the arguments are similar. If $p$ is inert, we compute that
\[
S_{d, p} = 1 + \sum_{0 \leq i \leq \frac{ \ell - 3}{2}} \frac{\chi(\p_z)^{g^{2i} \cdot (\ell - 1)/2}}{\N(p)^s} + 
\sum_{0 \leq i \leq \frac{ \ell - 3}{2}} \frac{\chi(\p_z)^{g^{2i + 1} \cdot (\ell - 1)/2}}{\N(p)^s},
\]
equal to the same expression as before. If $p$ is split, recall that by Proposition \ref{propsplit2} $\mfa_{\alpha}$ 
must be stable by $\tau$;
the relevant computation is
\[
\chi(\ov{\p_z \tau^{(\ell - 1)/2}(\p_z)})^{(\ell - 1)/2} = \chi(\ov{\p_z^{1 - g^{(\ell - 1)/2}}})^{(\ell - 1)/2} = \chi(\ov{\p_z})^{-1},
\]
and again we obtain the same results. For $p \in \calD_{\ell}$ the argument is similar, once again
considering all three cases and obtaining the same result.
\end{proof}

\smallskip

As mentioned in Remarks \ref{rem:ell3}, if $\ell>2[k:\Q]+1$, and in particular
if $k=\Q$ and $\ell\ge5$, we always have $\gothr^e(\b)=(1)$. The theorem
simplifies and gives the following:

\begin{corollary}\label{cormain} Keep the same notation, and assume that 
$\ell>2[k:\Q]+1$.
We have
$$\Phi_{\ell}(K,s)=\dfrac{\ell^{RU(K)}}{(\ell-1)\ell^{\frac{\ell}{\ell-1}[k:\Q]s}}
\prod_{\p\mid\ell}
\N(\p)^{-\frac{\ell-1-r(e(\p))}{\ell-1} s}\sum_{\b\in\mathcal{B}}
\dfrac{\NC(\b)^sP(\b,s)}{|(Z_{\b_z}/Z_{\b_z}^\ell)[T]|}\sum_{\chi\in\widehat{G_{\b}}}F(\b,\chi,s) \;,$$
where
$$F(\b,\chi,s)=\prod_{p\in\mathcal{D}'(\chi)} \left(1+\dfrac{\ell-1}{\N(p)^{s}}\right)\prod_{p\in\mathcal{D}\setminus\mathcal{D}'(\chi)} \left(1-\dfrac{1}{\N(p)^{s}}\right)\;.$$
\end{corollary}

In the general case, we now prove that the group $G_{\b}$ can be described in somewhat simpler terms, in terms
of the mirror field $K'$ of $K$. (See also Theorems \ref{thm:char_nf} and
Theorem \ref{thm:nakagawa} for a further characterization.)

\begin{proposition}\label{prop:iso_red}
There is a natural isomorphism
\[
\frac{ \Cl_{\mfb}(K_z)}{\Cl_{\mfb}(K_z)^{\ell}}[T] \rightarrow \frac{ \Cl_{\mfb'}(K')}{ \Cl_{\mfb'}(K')^{\ell} }[\tau - g],
\]
where $\mfb' = \mfb \cap K'$.

Moreover, using this isomorphism to regard a character $\chi$ of $\frac{ \Cl_{\mfb}(K_z)}{\Cl_{\mfb}(K_z)^{\ell}}[T]$
as a character $\widetilde{\chi}$ of $\frac{ \Cl_{\mfb'}(K')}{ \Cl_{\mfb'}(K')^{\ell} }[\tau - g]$,
the condition $\chi(\ov{\p_z}) = 1$ defining $\calD(\chi) \cap \calD'_{\ell}(\chi)$
is equivalent to the condition $\widetilde{\chi}(\ov{\p_{K'}}) = 1$ for the unique prime $\p_{K'}$ of $K'$ below $\p_z$.

\end{proposition}
\begin{proof} The first statement is also proved in \cite[Proposition 3.6]{CRST}, so we will be brief. As
$\tau^{(\ell - 1)/2} \tau_2$ acts trivially on $G_{\b}$, it can be checked that 
elements of $G_{\b}$ can be represented by an ideal of the form
$\mfa \tau_2 \tau^{(\ell - 1)/2} \mfa$, which is of the form $\mfa' \Z_{K_z}$ for some ideal
$\mfa'$ of $K'$. We therefore obtain a well-defined injective map
$\frac{ \Cl_{\mfb}(K_z)}{\Cl_{\mfb}(K_z)^{\ell}}[T] \rightarrow \frac{ \Cl_{\mfb'}(K')}{ \Cl_{\mfb'}(K')^{\ell} }[\tau - g],$
which may easily be shown to be surjective as well.

The latter statement follows because the condition $\widetilde{\chi}(\ov{\p_{K'}}) = 1$ is equivalent to
$\chi(\ov{\p_{K'} \Z_{K_z}}) = 1$, which is easily seen to be equivalent to $\chi(\ov{\p_z}) = 1$ for any splitting type
of $\p_z | \p_{K'}$.

\end{proof}
\section{Specialization to $k=\Q$}\label{sec_q}

We now specialize to $k = \Q$,
where we will obtain more explicit results. Henceforth we assume that
$K=\Q(\sqrt{D})$ is a quadratic field with discriminant $D$, with $r_2(D) = 0$ if $D > 0$ and $r_2(D) = 1$ if $D < 0$.

By definition, $\calB = \{1, (\ell), (\ell)^{\ell/(\ell - 1)} \}$ in the general case with $\ell \nmid D$, 
and $\calB = \{1, (\ell)^{1/2}, (\ell), (\ell)^{\ell/(\ell - 1)} \}$ in the special case or in the general
case with $\ell \mid D$. Equivalently we may write
\begin{equation}\label{eqn:poss_b}
\b_z \in \Big\{ \Z_{K_z}, \ (1-\z_\ell)^{(\ell-1)/2}\Z_{K_z}, 
\ (1-\z_\ell)^{\ell-1}\Z_{K_z}=\ell\Z_{K_z}, 
(1-\z_\ell)^\ell\Z_{K_z} \Big\}\;,
\end{equation}
{with the second entry removed in the former case.}
Throughout, $(-, -, -, -)$ will describe quantities depending on $\calB$,
with asterisks denoting `not applicable'.

\begin{proposition}\label{propzb}
We have that $|(Z_{\b_z}/Z_{\b_z}^\ell)[T]|$ is equal to $(1, \ast, \ell, \ell)$ or $(1, 1, \ell, \ell)$
for $\ell \nmid D$ or $\ell | D$ respectively, unless $\ell = 3$ in the special case, in which case
$|(Z_{\b_z}/Z_{\b_z}^3)[T]| = (1, 1, 1, 3)$.
\end{proposition}

\begin{proof} This follows from Theorem \ref{thmcb} and Lemma \ref{lem:27_es}.
In the general case we obtain
\[
|(Z_{\b_z}/Z_{\b_z}^\ell)[T]| = (0, \ast, 2, 2) - (0, \ast, 1, 1) = (0, \ast, 1, 1)\;,
\]
\[
|(Z_{\b_z}/Z_{\b_z}^\ell)[T]| = (0, 1, 2, 2) - (0, 1, 1, 1)  = (0, 0, 1, 1)\;,
\]
depending on whether $\ell \nmid D$ or $\ell | D$ respectively; in the special case we obtain
\[
|(Z_{\b_z}/Z_{\b_z}^\ell)[T]| = (0, 0, 1, 1)\;,
\]
\[
|(Z_{\b_z}/Z_{\b_z}^\ell)[T]| = (0, 0, 0, 1)\;,
\]
depending on whether $\ell \geq 5$ or $\ell = 3$ respectively.
\end{proof}

Recall by Lemma \ref{lemmirror} that the mirror field of $K=\Q(\sqrt{D})$ 
with respect to $\ell$ is the degree $\ell-1$ field
$K'=\Q(\sqrt{D}(\ze-\ze^{-1}))$. The following is immediate from
the results of Section \ref{sec:galois}:

\begin{lemma}\label{lemkp} Let $p$ be a prime different from $\ell$.
\begin{itemize}\item We have $p\in\calD$ if and only if
$p\equiv\leg{D}{p}\pmod{\ell}$.
\item In the general case, this is equivalent to 
$p$ splitting completely in $K'/\Q$.
\item In the special case with $\ell\equiv1\pmod4$, this is equivalent to
$p\equiv1\pmod{\ell}$.
\item In the special case with $\ell\equiv3\pmod4$, this is equivalent to
$p\equiv\pm1\pmod{\ell}$.\end{itemize}
\end{lemma}

We come now to the analogue of Theorem 3.2 of \cite{Coh-Tho}.
The case $\ell=3$, which is slightly different, is treated in loc.~cit.:

\begin{theorem}\label{thm_main_q} Assume that $\ell\ge5$ and let $K=\Q(\sqrt{D})$. We have
$$\Phi_\ell(K,s)=\dfrac{\ell^{r_2(D)}}{\ell-1}\sum_{\b\in\mathcal{B}}A_{\b}(s)\sum_{\chi\in\widehat{G_{\b}}}F(\b,\chi,s)\;,$$
where the $A_{\b}(s)$ are given by the
following table:

\medskip

\centerline{
\begin{tabular}{|c||c|c|c|c|}
\hline
Condition on $D$ & $A_{(1)}(s)$ & $A_{(\sqrt{(-1)^{(\ell-1)/2}\ell})}(s)$ & $A_{(\ell)}(s)$ & $A_{(\ell^{\ell/(\ell-1)})}(s)$\\
\hline\hline
$\ell\nmid D$ & $\ell^{-2s}$ & $0$ & $-\ell^{-2s-1}$ & $1/\ell$\\
\hline
$\ell\mid D$ & $\ell^{-3s/2}$ & $\ell^{-s}-\ell^{-3s/2}$ & $-\ell^{-s-1}$ & $1/\ell$\\
\hline
\end{tabular}}

\medskip

$$F(\b,\chi,s)=\prod_{p\equiv\leg{D}{p}\pmod{\ell},\ p\ne\ell}\left(1+\dfrac{\om_{\chi}(p)}{p^s}\right)\;,$$
where we set:
$$\om_\chi(p)=\begin{cases}
\ell-1&\text{\quad if $\chi(\p_z)=1$}\\
-1&\text{\quad if $\chi(\p_z)\ne1$}\;,\end{cases}$$
where as usual $\p_z$ is any ideal of $K_z$ above $p$.
\end{theorem}

\begin{proof} The computation is routine, given the following consequences of
our previous results:
\begin{itemize}
\item We have $k=\Q$ so $\ell^{\frac{\ell}{\ell-1}[k:\Q]s}=\ell^{\ell s/(\ell-1)}$.
\item The factor $\prod_{\p\mid\ell}\N(\p)^{...}$ is equal to
$\ell^{-(\ell-2)s/(\ell-1))}$ if $\ell\nmid D$ and to
$\ell^{-(\ell-3)s/(2(\ell-1))}$ if $\ell\mid D$. Multiplied by the first factor
this gives $\ell^{-2s}$ if $\ell\nmid D$ and $\ell^{-3s/2}$
if $\ell\mid D$.
\item We have $\ell^{RU(K)}=\ell^{r_2(D)}$ by Lemma \ref{lemu5}, with $r_2(D) := r_2(\Q(\sqrt{D}))$.
\item By Definitions \ref{defB} and \ref{def41}, 
we have 
$\NC(\b)=(1,*,\ell,\ell^2)$ and
$\NC(\b)=(1,\ell^{1/2},\ell,\ell^{3/2})$ for 
$\ell \nmid D$ and $\ell \mid D$ respectively.
\item As already mentioned, if $k=\Q$ and $\ell>3$ we have $\gothr^e(\b)=(1)$,
so the terms and conditions involving $\gothr^e(\b)$ disappear (in other
words we use Corollary \ref{cormain}).
\item By Lemma 
\ref{lemkp}, we have $p\in\calD$ if and only if 
$p\equiv\leg{D}{p}\pmod{\ell}$ and $p\ne\ell$, and $\calD_{\ell}=\emptyset$
when $\ell\ne3$ by what we have just said.
\item By Lemma \ref{lempb} and Definition \ref{defB} of $P(\b,s)$, when
$\ell\nmid D$ and $\ell \mid D$ respectively.
we have $P(\b,s)=(1,*,-\ell^{-s},1)$, and 
$P(\b,s)=(1,1-\ell^{-s/2},-\ell^{-s/2},1)$ for $\ell \nmid D$ respectively for the 
usual sequence of $\b$.
\item The values of 
$|(Z_{\b_z}/Z_{\b_z}^\ell)[T]|$ are given in Proposition \ref{propzb}.
\end{itemize}
\end{proof}

\begin{corollary}\label{lelldef} Assume that $\ell\ge5$, and set
$L_\ell(s)=1+(\ell-1)/\ell^{2s}$ if $\ell\nmid D$ and
$L_\ell(s)=1+(\ell-1)/\ell^s$ if $\ell\mid D$.
There exists a function $\phi_D(s)=\phi_{D,\ell}(s)$, holomorphic for 
$\Re(s)>1/2$, such that
$$\Phi_\ell(K, s)=
%\sum_{L\in\calFL(K)}\dfrac{1}{f(L)^s}=
\phi_D(s)+\dfrac{1}{(\ell-1)\ell^{1-r_2(D)}}L_{\ell}(s)\prod_{p\equiv\leg{D}{p}\pmod{\ell},\ p\ne\ell}\left(1+\dfrac{\ell-1}{p^s}\right)\;.$$
\end{corollary}

\begin{proof} Same as in Proposition 7.5 of \cite{Coh-Mor}:
The main term is the contribution of the
trivial characters, and $\phi_D(s)$ is the contribution of 
the nontrivial characters: we first regard each 
$\chi \in \widehat{G_{\b}}$ as a character of
$\frac{ \Cl_{\mfb'}(K')}{ \Cl_{\mfb'}(K')^{\ell} }$ by Proposition \ref{prop:iso_red} and then by 
setting $\chi$ equal to $1$ on the orthogonal complement of
$\frac{ \Cl_{\mfb'}(K')}{ \Cl_{\mfb'}(K')^{\ell} }[\tau - g]$. By the previous lemma, the primes occurring in the product
are precisely those for which $p$ is totally split in $K'$. Therefore, for each set of nontrivial characters 
$\chi, \chi^2, \dots, \chi^{\ell - 1} \in \widehat{G_b}$, the sum of products
$F(\b, \chi, s)$ may be written as $g(s) + \sum_{\chi} L(s, \chi)$, where $L(s, \chi)$ is the (holomorphic)
Hecke $L$-function associated to $\chi$, 
and $g(s)$ is a Dirichlet series supported on squarefull numbers, absolutely convergent and therefore 
holomorphic in $\Re(s) > 1/2$. Therefore $\phi_D(s)$ is holomorphic in $\Re(s) > 1/2$ as well. 
We also note that the product of the main term may similarly be written as
$h(s) + L(s, \omega_0)$, where $\omega_0$ is the trivial Hecke character, and $h(s)$ satisfies
the same properties as $g(s)$.

The $\ell = 3$ case is slightly different due to the nontriviality of $\gothr^e(\b)$;
see \cite{Coh-Mor}.
\end{proof}

%This brings us to our asymptotic formulas:

\begin{corollary}\label{corasym} Assume that $\ell\ge5$ and denote by $M_\ell(D;X)$ 
the number of $L\in\calFL(\Q(\sqrt{D}))$ such that $f(L)\le X$. 
Set $c_1(\ell)=1/((\ell-1)\ell^{1-r_2(D)})$, $c_2(\ell)=(\ell^2+\ell-1)/\ell^2$ when
$\ell\nmid D$ or $c_2(\ell)=2-1/\ell$ when $\ell\mid D$.
\begin{enumerate}\item In the general case, or in the special case with $\ell\equiv1\pmod4$, for any $\eps>0$ we have
$$M_\ell(D;X)=C_\ell(D)X+O_D(X^{1-\frac{2}{\ell + 3}+\eps})\text{\;, with}$$
$$C_\ell(D)=c_1(\ell)c_2(\ell)\Res_{s=1}\prod_{p\equiv\leg{D}{p}\pmod{\ell}}\Bigl(1+\frac{\ell-1}{p^s}\Bigr)\;,$$
and in the special case the product is
equivalently over $p\equiv1\pmod{\ell}$.

\item In the special case with $\ell\equiv3\pmod4$, for any $\eps>0$ we have
$$M_\ell(D;X)=C_\ell(D)(X\log(X)+C'_\ell(D))+O_D(X^{1-\frac{2}{\ell + 3}+\eps})\text{\;, with}$$
$$C_\ell(D)=c_1(\ell)c_2(\ell)\lim_{s\to1^+}(s-1)^2\prod_{p\equiv\pm1\pmod\ell}\Bigl(1+\frac{\ell-1}{p^s}\Bigr)\;,$$
and $C'_\ell(D)$ can also be given explicitly if desired.
\end{enumerate}\end{corollary}

\begin{proof} The result follows by the same proof as in \cite{Coh-Mor},
with $C_\ell(D)$ equal to the residue at $s=1$
of $\Phi_\ell(K,s)$. 

We briefly recall how to obtain the error term. By the proof of Corollary \ref{lelldef}, it equals 
(up to an implied constant depending on $D$ and $\ell$)
the error made in estimating partial sums of Hecke $L$-functions
of degree $\ell -1$.
We carry this out in the standard way, subject to the limitation that we may not shift any contour to $\Re(s) \leq 1/2$.
We have by Perron's formula, for each Hecke $L$-function $\xi(s) = \sum_n a(n) n^{-s}$ and any $c > 1$,
\[
\sum_{n < X} a(n) = \frac{1}{2 \pi i} \int_{c - i \infty}^{c + i \infty} \xi(s) \frac{X^s}{s} ds\;,
\]
and we shift the portion of the contour from $c - iT$ to $c + iT$ to $\Re(s) = \sigma$ 
for $\sigma \in (1/2, 1)$ and $T> 0$ to be determined. 
By convexity we have 
$|\xi(s)| \ll T^{\frac{\ell - 1}{2} (1 - \sigma + \eps)}$, and choosing $c = 1 + \eps, \ \sigma = 1/2 + \eps$ 
our integral is $\ll 
T^{ \frac{\ell - 1}{2}( \frac{1}{2} + 2 \eps)} X^{\frac{1}{2} + \eps} + \frac{X^{1 + 2 \eps}}{T}$;
then choosing $T = X^{\frac{2}{\ell + 3}}$ we obtain an error term of $X^{1 - \frac{2}{\ell + 3}+\eps}$.
\end{proof}

In a separate paper by the first author \cite{Cohasym}, one explains how to compute the constants $C_\ell(D)$
to high accuracy ($100$ decimal digits, say) for reasonably small values of $|D|$. For example, we have
\[C_3(-3)  = 0.0669077333013783712918416\cdots,\quad C_3(-4) = 0.1362190676241212841449867\cdots.\\
%C_3(-7)&=0.1533459546528706230534532\cdots,\quad C_3(5)=0.0818840074459636358232037\cdots,\\
%C_3(8)&=0.0697794325982058645585930\cdots,\quad C_3(12)=0.0803828977056554045622405\cdots,\\
%C_5(-3)&=0.0507853244497800993782016\cdots,\quad C_5(-4)=0.0533779051631195981220655\cdots,\\
%C_5(-7)&=0.0646534750523185710598435\cdots,\quad C_5(5)=0.0203781870559037146558936\cdots,\\
%C_5(8) &=0.0134747747475919140437863\cdots,\quad C_5(12)=0.0154777556594427976114120\cdots,\\
%C_7(-3)&=0.0296332163247300745247219\cdots,\quad C_7(-4)=0.0292582526699757840365146\cdots,\\
%C_7(-7)&=0.0121052634214512298018579\cdots,\quad C_7(5)=0.0064676733264714100259068\cdots,\\
%C_7(8)&=0.0057454974330245481725806\cdots,\quad C_7(12)=0.0035078864947419330918974\cdots.
\]

%It is natural to ask whether we can obtain estimates for 
%the count $N_{D_\ell}^{\pm}(X)$ of
%degree $\ell$ fields $L$ with Galois group $D_\ell$, $|\Disc(L)| \leq X$, and whose quadratic
%resolvent is respectively real or imaginary,  A plausible guess is that for some $C_{\ell} > 0$ we have
%\begin{equation}\label{eqn:dl_guess}
%N_{D_\ell}^-(X)\sim C_\ell X^{2/(\ell-1)}\text{\quad and\quad}N_{D_\ell}^+(X)\sim \dfrac{C_\ell}{\ell}X^{2/(\ell-1)}\;.
%\end{equation}

%By Davenport--Heilbronn this is known for
%$\ell=3$ with $C_3=1/(4\zeta(3))$, but it is unclear how to recover
%this value of $C_3$, even heuristically, from our work. As such we still seem to be far from a proof
%of \eqref{eqn:dl_guess}.

\section{Study of the Groups $G_{\b}$}\label{sec:gb}

In this section, where we continue to assume that $k = \Q$ and also assume that $\ell \geq 5$,
we study the groups $G_{\b}$
appearing in Theorem \ref{thm_main_q}. Much of this was carried out 
in our paper \cite{CRST} with Rubinstein-Salzedo, and 
we give only a brief account of those
results which are proved there.

We are indebted to 
Hendrik Lenstra for help in this section. 

\smallskip
We recall a few of the important notations
used previously:
\begin{itemize}
\item $K_z$ is an abelian extension of $\Q$ containing the $\ell$th roots of unity, 
with $G = \Gal(K_z/\Q) = \langle \tau, \tau_2 \rangle$
or $\langle \tau \rangle$ in the general and special cases respectively.
\item As in Proposition \ref{propbij1}, $N_z=K_z(\root\ell\of\al)$ is a cyclic extension, for which we wrote 
$\alpha \Z_{K_z} = \q^{\ell} \prod_{0 \leq i \leq \ell - 2} \mfa_i^{g^i}$ and (in Proposition \ref{propsplit2})
$\prod_{0 \leq i \leq \ell - 2} \mfa_i = \mfa_{\alpha} \Z_{K_z}$ for an ideal $\mfa_{\alpha}$ of $K$.
\item We recall the possibilities for $\b$ (equivalently, $\b_z$) from  \eqref{eqn:poss_b},
and we continue to use the notation $(-, -, -, -)$ for quantities depending on $\b$.
\end{itemize}

For any $\b$ as in \eqref{eqn:poss_b} we define $\b^* := (1 - \z_\ell)^\ell/\b_z$.

\begin{proposition}\label{prop:kum_hec}
With the notation above, we have $\f(N_z/K_z)\mid\b_z$ if and only 
if $\ov{\al}\in S_{\b^*}(K_z)$.
\end{proposition}

\begin{proof} This is very classical, and essentially due to Kummer and
Hecke: for instance, by Theorem 3.7 of \cite{CoDiOl3} we have
$$\f(N_z/K_z)=(1-\z_\ell)^\ell\a_{\al}/\prod_{\p_z\mid\ell,\ \p_z\nmid\a_{\al}}\p_z^{A_{\al}(\p_z)-1}\;.$$
Thus, since $\a_{\al}$ is coprime to the product then 
$\f(N_z/K_z)\mid(1-\z_\ell)^\ell$ if and only if $\a_{\al}=\Z_K$, i.e.,
if and only if $\al$ is a virtual unit. If this is the case, then $\f(N_z/K_z)\mid\b_z$ if
and only if the product is a multiple of $(1-\z_\ell)^\ell/\b_z=\b^*$, and by the definition of $A_{\al}$ and
the congruence in Proposition \ref{propap}, this is equivalent to the
solubility of the congruence $x^\ell/\al\equiv1\pmod{^*\b^*}$, hence to
$\ov{\al}\in S_{\b^*}(K_z)$.\end{proof}

\begin{theorem}\cite[Corollary 3.2]{CRST}\label{thm:main_kummer} Writing 
$C_{\b}:=\Cl_{\b_z}(K_z)/\Cl_{\b_z}(K_z)^\ell$, so that $G_{\b} = C_{\b}[T]$, and $\bmu_\ell$ for the group of $\ell$th roots of 
unity,
there exists a perfect, $G$-equivariant pairing of 
$\F_\ell[G]$-modules
$$C_{\b}\times S_{\b^*}(K_z)\mapsto\bmu_\ell\;.$$
\end{theorem}

\begin{proof} This is the Kummer pairing: given $\ov{\a} \in C_{\b}$, let $\sigma_{\a}$
denote its image under the Artin map; given $\ov{\al}\in S_{\b^*}(K_z)$, let $\al$ be any lift;
then define the pairing by 
$(\ov{\a},\ov{\al})\mapsto \sigma_{\a}(\root\ell\of\al)/\root\ell\of\al\in\bmu_\ell.$
\end{proof}

\begin{corollary}\cite[Corollary 3.3 (in part)]{CRST}\label{cor:gb_pairing} In the general case, where $T=\{\tau-g,\tau_2+1\}$,
define $T^*=\{\tau-1,\tau_2+1\}$, and in the special case, where $T=\{\tau+g\}$, define
$T^*=\{\tau+1\}$. Then
we have a perfect pairing
$$G_{\b}\times S_{\b^*}(K_z)[T^*]\mapsto\bmu_\ell\;.$$
In particular, we have
$$|G_{\b}| = |S_{\b^*}(K_z)[T^*]|\;.$$

\end{corollary}

\begin{proof} 
Recalling that $\tau(\z_\ell) = \z_\ell^g$, for any $j$ the preceding corollary yields 
a perfect pairing 
$$C_{\b}[\tau-g^j]\times S_{\b^*}(K_z)[\tau-g^{1-j}]\mapsto\bmu_\ell\;.$$

We conclude by taking $j = 1$ and $j = (\ell + 1)/2$ in the general and special cases respectively.
\end{proof}

\begin{proposition}\label{prop:sc_res}
In the special case, we have
$\Cl(\Q_z)/\Cl(\Q_z)^{\ell}[\tau + 1] = \{1\}\;.$
\end{proposition}
\begin{proof}
We first show that there exists
an isomorphism
\[
\Cl(\Q_z)/\Cl(\Q_z)^{\ell}[\tau + 1] \isom \Cl(K)/\Cl(K)^{\ell}[\tau + 1]\;.
\]
By Lemma \ref{lem24} (which also applies to $t = \tau + 1$), the left side consists of those classes which may
be represented by ideals of the form $\N_{\Q_z/K}(\a) / \tau(\N_{\Q_z/K}(\a))$. We therefore obtain a well-defined,
injective map to $\Cl(K)/\Cl(K)^{\ell}[\tau + 1]$. Any ideal in the target space may be represented by an ideal of the form
$\c / \tau(\c)$, which is equivalent to $(\c / \tau(\c))^{(\ell - 1)^2}$, and $\c^{(\ell - 1)^2} = \N_{\Q_z/K}(\c^{2 (\ell - 1)} \Z_{\Q_z})$, so that the map
is surjective as well.

Now it suffices to show that $\ell \nmid h(\pm \ell)$, where $h(D)$ denotes
the class number of $\Q(\sqrt{D})$, and this follows from the fact that $h(\pm\ell)<\ell$ for all prime $\ell$.
\end{proof}

\begin{remark}
For $\ell \equiv 3 \pmod 4$ it is also possible to prove the proposition via 
the Herbrand-Ribet theorem and a congruence for Bernoulli numbers.
\end{remark}

Now suppose that $\ell \equiv 1 \pmod 4$. Then 
the {\itshape Ankeny-Artin-Chowla conjecture} (AAC) \cite{AAC, Mordell} states that if $\epsilon = (a + b\sqrt{\ell})/2$
is the fundamental unit of $\Q(\sqrt{\ell})$, then $\ell \nmid b$. We will use the statement of the conjecture
directly, but we note that 
Ankeny and Chowla \cite{AC} and Kiselev \cite{Kiselev} proved that it
is equivalent to the condition $\ell \nmid B_{(\ell - 1)/2}$, which is
trivially true if $\ell$ is a regular prime, a result first proved by
Mordell \cite{Mordell}. It has been verified for $\ell \leq 2 \cdot 10^{11}$ 
by van der Poorten, te Riele, and Williams \cite{PRW}, but as mentioned
in the introduction, on heuristic grounds it it probably false.

\begin{lemma}\label{lem:aac_ck}
\begin{enumerate}
\item
If AAC is true for $\ell$, the congruence
$x^\ell\equiv\eps\pmod{(1-\z)^k\Z_{\Q_z}}$
is solvable for $k=(\ell-1)/2$, and not for any larger value of $k$.

\item
If AAC is false for $\ell$, then this congruence is soluble for all $k$.

\end{enumerate}
\end{lemma}

\begin{proof} First assume AAC, and write $\eps=(a+b\sqrt{\ell})/2$ with $a, b \in \Z$.
Note first that $(1-\z)^{(\ell-1)/2}\Z_{\Q_z}=\sqrt{\ell}\Z_{\Q_z}$,
and $\eps\equiv a/2\equiv (a/2)^\ell\pmod{\sqrt{\ell}\Z_{\Q_z}}$, so the congruence is solvable with $k=(\ell-1)/2$.

Conversely for each $x \in \Z_{\Q_z}$ we have
$x^\ell\equiv m\pmod{\ell}$ for some $m \in \Z$, so that a solution to 
$x^\ell\equiv\eps\pmod{\sqrt{\ell}(1-\z)\Z_{\Q_z}}$ would yield
$a+b\sqrt{\ell}\equiv 2m\pmod{\sqrt{\ell}(1-\z)\Z_{\Q_z}}$. We thus have
$a\equiv 2m\pmod{\sqrt{\ell}}$ and hence $a\equiv 2m\pmod{\ell}$, yielding
$b\equiv0\pmod{(1-\z)\Z_{\Q_z}}$ and so $\ell\mid b$, violating AAC and proving (1).

To obtain (2), observe that the congruence may trivially be solved for $k = 3(\ell - 1)/2$ with $x \in \Z$, 
after which a Newton-Hensel
iteration as in \cite[Lemma 10.2.10]{Coh1} settles the matter.
\end{proof}

We now return to the groups $S_{\b^*}(K_z)[T^*]$.
\begin{proposition}\label{prop:count_sel}
\begin{enumerate}
\item In the general case we have $S_{\b^*}(K_z)[T^*]\isom S_{\b^*\cap K}(K)$.
\item In the special case with
$\ell\equiv3\pmod4$, we have $S_{\b^*}(K_z)[T^*]=\{1\}$ for all $\b$.
\item In the special case with $\ell\equiv1\pmod4$, if the Ankeny-Artin-Chowla conjecture is true for $\ell$, then we have
$|S_{\b^*}(K_z)[T^*]|= (1,  1, \ell, \ell)$ for $\b$ as in \eqref{eqn:poss_b}. If Ankeny-Artin-Chowla is false for $\ell$, then
we have instead $|S_{\b^*}(K_z)[T^*]|= (\ell,  \ell, \ell, \ell)$.
\end{enumerate}
\end{proposition}

\begin{proof} (1) \cite[Proposition 3.4]{CRST}. We have an injection
$S_{\b^*\cap K}(K) \longhookrightarrow S_{\b^*}(K_z)[\tau-1]
$ which we prove is surjective by Hilbert 90 and some elementary computations, yielding an isomorphism
$S_{\b^*}(K_z)[\tau-1,\tau_2+1] \isom S_{\b^*\cap K}(K)[\tau_2+1]$. 
Furthermore, we have
$$S_{\b^*\cap K}(K)=S_{\b^*\cap K}(K)[\tau_2+1]\oplus S_{\b^*\cap K}(K)[\tau_2-1]\;,$$
and we argue that $S_\ell(K)[\tau_2-1]$ is trivial (and a fortiori all the
$S_{\b^*\cap K}[\tau-1]$), again using Hilbert 90.

\smallskip

(2) and (3). Assume now that we are in the special case, so that $K_z=\Q_z=\Q(\z_\ell)$. By Proposition \ref{prop:sc_res} we have
$(\Cl(K_z)/\Cl(K_z)^\ell)[\tau+1]=\{1\}$, so that by Lemma \ref{lemseq} we have
$S_\ell(K_z)[T^*]\isom (U(K_z)/U(K_z)^\ell)[\tau-g^{(\ell-1)/2}]$. By Theorem 2.3 of \cite{CoDiOl3} we deduce that $S_\ell(K_z)[T^*]$ is
trivial if 
$\ell\equiv3\pmod4$, $\ell\ne3$, 
and when $\ell\equiv1\pmod4$ that it is an $\F_\ell$-vector space of dimension
$1$. If $\eps$ is a fundamental unit of $K=\Q(\sqrt{\ell})$, then since 
$\tau$ acts on $\eps$ as Galois conjugation of $K/\Q$, we have
$\eps\tau(\eps)=\N_{K/\Q}(\eps)=\pm1$, which is an $\ell$th power. It
follows that $S_{\ell}(K_z)[T^*]=\{\eps^j,\ j\in\F_\ell\}$.

The sizes of the ray Selmer groups are then established by Lemma \ref{lem:aac_ck}.
\end{proof}

\begin{remark}\label{rk:fail_aac2}
The assumption that $\ell \neq 3$ is required when applying Theorem 2.3 of \cite{CoDiOl3}, and
indeed (2) of the proposition is false for $\ell =3$ (see Proposition 7.3 of \cite{Coh-Mor}).

\end{remark}
This proposition, in combination with Corollary \ref{cor:gb_pairing}, gives the size of $|G_{\b}|$
in the special case, with possible exceptions $\ell \equiv 1 \pmod 4$ larger than $2 \cdot 10^{11}$.
In the general case we have the following:

\begin{corollary}\cite[Corollary 3.5]{CRST} Assume that we are in the general case.
\begin{enumerate}\item We have a canonical isomorphism 
$G_{\b}\isom\Hom(S_{\b^*\cap K}(K),\bmu_\ell)$.
\item In particular
$$|G_{\b}|=\ell^{r(\b)}\text{\quad with\quad}r(\b)=1-r_2(D)-z(\b)+\rk_\ell(\Cl_{\b^*\cap K}(K))\;,$$
with $z(\b)=(2,1,0,0)$ respectively, the second case occurring only if $\ell | D$.
\item In particular still, if $D<0$ and $\ell\nmid h(D)$ then $G_{\b}$ is trivial 
for all $\b\in\calB$.
\end{enumerate}
\end{corollary}

\begin{proof} (1) is immediate.
Lemma \ref{lemseq}, and Proposition 2.12 of \cite{CoDiOl3}, the proofs of which adapt to $K$
without change, yield
$$|S_{\b_1}(K)||Z_{\b_1}/Z_{\b_1}^\ell|=\ell^{1-r_2(D)}|\Cl_{\b_1}(K)/\Cl_{\b_1}(K)^\ell|\;,$$
where $Z_{\b_1}=(\Z_K/\b_1)^*$ and $\b_1=\b^*\cap K$. This gives (2) with
$z(\b)=\dim_{\F_\ell}(Z_{\b_1}/Z_{\b_1}^\ell)$, and to finish we compute for $\b$ as in
\eqref{eqn:poss_b}:
\begin{itemize}
\item If $\ell\nmid D$, we have $\b^*\cap K= (\ell^2\Z_K, \ast, \ell\Z_K, \Z_K)$.
\item If $\ell\mid D$ with $\ell\Z_K=\p_\ell^2$, we have
$\b^*\cap K= (\p_\ell^3, \p_\ell^2, \p_\ell, \Z_K)$.
\end{itemize}
\end{proof}

Note that (3) is a generalization of Proposition 7.7 of \cite{Coh-Mor}.

Since the triviality of $G_{\b}$ for all $\b$ is equivalent to the vanishing of
the ``remainder term'' $\phi_D(s)$ of Corollary \ref{lelldef},
we conclude that $\Phi_{\ell}(K, s)$ is given by a single Euler product in a wide class of examples:
\begin{corollary}\label{cor:34_simple} Assume that $\ell\ge5$, $D<0$, and that either we are in
the special case (so that $\ell\equiv3\pmod4$), or that we are in the general
case with $\ell\nmid h(D)$. Then we have
$$\sum_{L\in\calFL(K)}\dfrac{1}{f(L)^s}=-\dfrac{1}{\ell-1}+\dfrac{1}{\ell-1}L_{\ell}(s)\prod_{p\equiv\leg{D}{p}\pmod{\ell},\ p\ne\ell}\left(1+\dfrac{\ell-1}{p^s}\right)\;,$$
where $L_\ell(s)$ is as above.\end{corollary}

Note that for $\ell = 3$, which we have excluded here, 
the possible nontriviality of $\gothr^e(\b)$ forces us 
to also distinguish between $D\equiv3$ and 
$D\equiv6\pmod{9}$.

\medskip

{\bf Examples with $\ell=5$:}

$$\sum_{L\in\calF_5(\Q(\sqrt{-1}))}\dfrac{1}{f(L)^s}=-\dfrac{1}{4}+
\dfrac{1}{4}\left(1+\dfrac{4}{5^{2s}}\right)
\prod_{p\equiv\pm1\pmod{20}}\left(1+\dfrac{4}{p^s}\right)\;.$$

$$\sum_{L\in\calF_5(\Q(\sqrt{-15}))}\dfrac{1}{f(L)^s}=-\dfrac{1}{4}+
\dfrac{1}{4}\left(1+\dfrac{4}{5^s}\right)
\prod_{p\equiv\pm1\pmod{30}}\left(1+\dfrac{4}{p^s}\right)\;.$$

\section{Transformation of the Main Theorem}\label{sec_trans}

We now prove, as we did in \cite{Coh-Tho} for the case of $\ell = 3$, that 
the characters of $G_{\b}$ appearing in Theorem \ref{thm:main} can be given a simpler
description, in terms of the splitting of primes in degree $\ell$ extensions of $k$.
Our main result along these lines extends Theorem 4.1 of \cite{Coh-Tho} and Proposition 3.7 of \cite{CRST}, and does not assume
that $k = \Q$, and thus is new even for $\ell = 3$. 

For the case $k = \Q$ we will further specialize the result and obtain an explicit
formula, relying (in the general case) on the results of \cite{CRST}.
We will assume that we are
in either the general case or in the special case with $\ell \equiv 1 \pmod 4$. Recall that in the special case
with $k = \Q$, $\ell \equiv 3 \pmod 4$, and $\ell > 3$, $G_{\b}$ is trivial and Corollary \ref{cor:34_simple} already gives a simple
description of $\Phi_\ell(K, s)$. For simplicity's sake we will omit the special case with $k \neq \Q$, $\ell \equiv 3 \pmod 4$;
as we will see below the group theory would work out a bit differently. 

Recall that the {\itshape Frobenius group}
$F_{\ell} = C_{\ell} \rtimes C_{\ell - 1}$ is the non-abelian group of order $\ell (\ell - 1)$ given as
\[
\langle \tau, \sigma \ : \ \tau^{\ell - 1} = \sigma^{\ell} = 1, \tau \sigma \tau^{-1} = \sigma^h \rangle\;,
\]
for any primitive root $h \pmod{\ell}$. As may be easily checked, $C_{\ell - 1}$ is not normal in $F_{\ell}$,
nor is any nontrivial subgroup of $C_{\ell - 1}$; moreover, there are $\ell$ subgroups isomorphic to $C_{\ell - 1}$,
generated by $\tau \sigma^i$ for $0 \leq i \leq \ell - 1$, and all of these subgroups are conjugate.
We will say that a degree $\ell$ field extension $E/k$ is an {\itshape $F_{\ell}$-extension} if its Galois
closure has Galois group $F_{\ell}$ over $k$.

Now, let $K, K_z, \tau, \tau_2$ be defined as before. In the general case recall that $K'$ was defined to be the mirror field of $K$,
e.g., the subfield of $K_z$ fixed by $\tau^{(\ell - 1)/2} \tau_2$; in the special case write $K' = K_z = k_z$.
We chose $\tau \in \Gal(k_z/k)$ and a primitive root $g \pmod \ell$ with $\tau(\ze) = \ze^g$. In the general case
$\tau$ lifts uniquely to an element of $\Gal(K_z/K)$ and restricts to a unique element of $\Gal(K'/k)$, so in either case
the choice of $g \pmod \ell$ uniquely determines $\tau \in \Gal(K'/k)$.

\begin{theorem}\label{thm:char_nf}
Assume, if $\ell \equiv 3 \pmod 4$, that we are in the general case. For each $\b \in \calB$ (as in 
Theorem \ref{thm:main}), there exists a bijection between the following sets:

\begin{itemize}
\item
Characters $\chi \in \widehat{G_{\mfb}}$, up to the equivalence relation $\chi \sim \chi^a$ for each $a$ coprime to $\ell$.
\item
Subgroups of index $\ell$ of $G_{\mfb}$.
\item
$F_\ell$-extensions $E/k$ (up to isomorphism), whose Galois closure $E'$ 
contains $K'$ and whose conductor $\mff(E'/K')$ divides $\mfb' = \mfb \cap K$, and such that
$\tau \sigma \tau^{-1} = \sigma^g$ for $\tau \in \Gal(K'/k)$ as described above and 
any generator $\sigma$ of $\Gal(E'/K')$.
\end{itemize}
Moreover, for each corresponding pair $(\chi, E)$ and
each prime $p \in \calD \cup \calD_{\ell}$, the following is true: we have $p \in \calD'(\chi) \cup \calD_{\ell}'(\chi)$
if and only if $p$ is totally split in $E$; equivalently,
$p \not \in \calD'(\chi) \cup \calD_{\ell}'(\chi)$ if and only if $p$ is totally inert or totally ramified in $E$.

\end{theorem}

%Recall (Definition \ref{defcald}) that $\calD \cup \calD_{\ell}$ was defined in terms of splitting
%conditions in $K_z/k$, so that this theorem describes each Euler 
%factor in $\Phi_{\ell}(K, s)$ in terms of splitting conditions in a fixed set of number fields.

\begin{proof}
The proof borrows heavily from those of Proposition 4.1 of \cite{Coh-Tho} and Proposition 3.7 of \cite{CRST}. 

The correspondence between the first two sets is immediate: $G_{\b}$ is elementary $\ell$-abelian, and characters
correspond to their kernels.

By Proposition \ref{prop:iso_red}, regard $G_{\b}$ as $\frac{ \Cl_{\mfb'}(K')}{ \Cl_{\mfb'}(K')^{\ell} }[\tau \mp g]$,
where the sign is $-$ in the general case and $+$ in the special case. If we set
$G_{\b}' = \Cl_{\b'}(K')/\Cl_{\b'}(K')^{\ell}$, by Lemmas \ref{lemidem} and \ref{lem24} we
have the orthogonal decomposition
$G_{\b}' = G_{\b} \times G_{\b}'',$
where $G_{\b}''$ is the direct sum of all of the other eigenspaces for the actions of $\tau$. 
Thus, subgroups of $G_{\b}$ of index $\ell$ correspond to 
subgroups $B$ of $\Cl_{\b'}(K')$ of index $\ell$ containing $G_{\b}''$.

By class field theory, there exists a unique abelian extension $E'/K'$, with Galois group $C_{\ell}$ and
conductor dividing $\b'$,
for which the Artin map induces an isomorphism $\Cl_{\b'}(K') / B \isom \Gal(E'/K')$. The uniqueness
forces $E'$ to be Galois over $k$; here we use that $\b'$, 
$B$, and
$\Cl_{\b'}(K')$ are preserved by $\Gal(K'/k)$. For each fixed $\b$,
 we obtain
 a different $E'$ for each $B$.
 
 Because the action of $\Gal(K'/k)$ on $\Cl_{\b'}(K') / B_{\chi}$ matches its conjugation action
 on $\Gal(E'/K')$, we have 
 \begin{equation}\label{eqn:gal_semi}
 \Gal(E'/k) = \langle \tau, \sigma \ : \tau^{\ell - 1} = \sigma^{\ell} = 1,
\tau \sigma \tau^{-1} = \sigma^{\pm g} \rangle \isom F_{\ell}\;,
 \end{equation}
and we take $E$ to be the fixed field of $\langle \tau \rangle$ (or, alternatively, of any conjugate subgroup).
Note that $-g$ is not a primitive root if $\ell \equiv 3 \pmod 4$, so that in the special case
with  $\ell \equiv 3 \pmod 4$
the group
\eqref{eqn:gal_semi} contains $\tau^{(\ell - 1)/2}$ in its center and is not isomorphic to $F_\ell$.

\smallskip
It must finally be proved that whether $p \in \calD'(\chi)$ or not is determined by its splitting in $E.$
Proposition \ref{propkp} or Corollary \ref{corsplit1} implies that $\calD \cup \calD_{\ell}$ is precisely
the set of primes $p$ which split completely in $K'/k$, and by definition 
$\calD'(\chi) \cup \calD'_{\ell}(\chi)$ is the set of primes $p \in \calD \cup \calD_{\ell}$ for which
one (equivalently, all) of the primes $\p_{K'}$ of $K'$ above $p$ split completely in $E'$.
If $\p_{K'}$ splits completely in $E'$, then so does $p$, so $p$ also splits completely in $E/k$. Conversely, if any $\p_{K'}$ 
is completely ramified or inert
in $K_z$, then $p$ must also do the same in each $E$, since ramification and inertial degrees are
multiplicative and $[E' : E] = \ell - 1$.
\end{proof}

For $\ell =3$ and $k = \Q$ in the general case,
in \cite{Coh-Tho} we further applied a theorem of Nakagawa to give a precise description of {\itshape all}
the extensions $E/\Q$ occurring in the statement of Theorem \ref{thm:char_nf} in terms of their
discriminants. Using this, we obtained the formula 
\begin{equation}\label{eqn:ct3}
\frac{2}{3^{r_2(D)}} \Phi_3(Q(\sqrt{D}), s) = M_1(s) \prod_{ \big( \frac{-3D}{p} \big) = 1}
\bigg(1 + \frac{2}{p^s} \bigg) + \sum_{E \in \mathcal{L}_3(D)} M_{2, E}(s)
\prod_{ \big( \frac{-3D}{p} \big) = 1}
\bigg(1 + \frac{\omega_E(p)}{p^s} \bigg)\;,
\end{equation}
where: $\mathcal{L}_3(D)$ is the set of all cubic fields of discriminant $-D/3$, $-3D$,
and $-27D$; $\omega_E(p)$ is $2$ or $-1$ depending on whether $p$ is split or inert in $E$, as in
Theorem \ref{thm:char_nf}; and $M_1(s)$ and $M_{2, E}(s)$ are $3$-adic factors (a sum of the appropriate
$A_{\mfb}(s)$).

\medskip

\subsection{Explicit computations for $k=\Q$ in the special case} 
For $\ell=3$, we have the following explicit formula (corresponding to pure cubic fields), which was previously proved in \cite{Coh-Mor}.

$$\sum_{L\in\calF_3(\Q(\sqrt{-3}))}\dfrac{1}{f(L)^s}=-\dfrac{1}{2}+\dfrac{1}{6}\left(1+\dfrac{2}{3^s}+\dfrac{6}{3^{2s}}\right)\prod_{p\ne3}\left(1+\dfrac{2}{p^s}\right)+\dfrac{1}{3}\prod_{p}\left(1+\dfrac{\om_E(p)}{p^s}\right)\;,$$
where $E$ is the cyclic cubic field defined by $x^3-3x-1=0$ of discriminant $3^4$, and
$$\om_E(p)=\begin{cases}
-1&\text{\quad if $p$ is inert or totally ramified in $E$\;,}\\
2&\text{\quad if $p$ is totally split in $E$ \ (equivalently, $p\equiv\pm1\pmod9$)\;.}
\end{cases}$$

\medskip

For $\ell\equiv3\pmod4$ and $\ell > 3$, a generalization was proved in Corollary \ref{cor:34_simple}.
For $\ell\equiv1\pmod4$, the generalization is more complicated due to the nontriviality of
$G_\b$. Define
a polynomial
$$P(x)= -2i T_{\ell}(ix/2) = \sum_{k = 0}^{(\ell-1)/2} \ell\dfrac{(\ell-k-1)!}{k!(\ell-2k)!}x^{\ell-2k}\;.$$

Here $T_{\ell}(x)$ is the Chebyshev polynomial of the first kind, satisfying
\begin{equation}\label{eqn:def_cheby}
P(x - x^{-1}) = x^{\ell} - x^{-\ell}\;,
\end{equation}
so that $x^{-1} P(x)$ is the minimal polynomial of $\ze - \ze^{-1}$.

\begin{proposition}\label{prop:special_explicit} Assume that $\ell\equiv1\pmod4$ 
satisfies the Ankeny-Artin-Chowla conjecture, and let $\eps$ be a 
fundamental unit of $\Q(\sqrt{\ell})$. Then we have
$$\sum_{L\in\calF_\ell(\Q(\sqrt{\ell}))}\dfrac{1}{f(L)^s}=-\dfrac{1}{\ell-1}+\dfrac{1}{\ell(\ell-1)}\left(1+\dfrac{\ell-1}{\ell^s}\right)\prod_{p\equiv1\pmod\ell}\left(1+\dfrac{\ell-1}{p^s}\right)+\dfrac{1}{\ell}\prod_{p}\left(1+\dfrac{\om_E(p)}{p^s}\right)\;,$$
where $E$ is the $F_\ell$-field defined by $P(x)-\Tr(\eps)=0$ of 
discriminant $\ell^{(3\ell-1)/2}$, and
$$\om_E(p)=\begin{cases}
-1&\text{\quad if $p$ is inert or totally ramified in $E$\;,}\\
\ell-1&\text{\quad if $p$ is totally split in $E$\;,}\\
0&\text{\quad otherwise.}\end{cases}$$
If $\ell \equiv 1 \pmod 4$ does not satisfy the Ankeny-Artin-Chowla conjecture, we have the same formula, but
with $\Disc(E) = \ell^{\ell - 2}$ and $\omega_E(\ell) = \ell - 1$.

\end{proposition}

\begin{corollary}\label{cor:ac_exist}
Let $\ell \equiv 1 \mod 4$. Then there exist $D_{\ell}$-fields ramified only at $\ell$ if and only if
the Ankeny-Artin-Chowla conjecture is false for $\ell$.
\end{corollary}
\begin{proof}
Immediate; for $\ell$ not satisfying the conjecture, the field is unique and has discriminant $\ell^{\frac{3(\ell - 1)}{2}}$.
\end{proof}
\begin{remark}
This corollary recovers and strengthens a result of Jensen and Yui \cite[Theorem I.2.2]{JY}, who proved that 
if $\ell \equiv 1 \pmod 4$ is regular, then there are no $D_{\ell}$-fields with discriminant a power of $\ell$. (This can also
be seen for $\ell \equiv 3 \pmod 4$ from Corollary \ref{cor:34_simple}.)

The connection to the Ankeny-Artin-Chowla conjecture was previously observed by Lemmermeyer \cite{lemmer}, who
suggested that a proof of Corollary \ref{cor:ac_exist} may exist somewhere
in the literature.
\end{remark}

Before beginning the proof of Proposition \ref{prop:special_explicit} we establish the following:

\begin{lemma}\label{lem:disc_E} We have
$$\Disc(N_z)=\begin{cases}
\ell^{(3\ell^2-2\ell-3)/2}&\text{\quad if AAC is true,}\\
\ell^{\ell(\ell-2)}&\text{\quad if AAC is false.}
\end{cases}$$
In addition, in the extension $N_z/\Q_z$ the prime ideal $(1-\z)\Z_{\Q_z}$
is totally ramified if AAC is true and totally split otherwise.
\end{lemma}

\begin{proof} The field $N_z$ is a Kummer extension of $K_z=\Q_z$ with 
defining equation $x^\ell-\eps=0$, so that
$$\Disc(N_z)=\pm\N_{\Q_z/\Q}(\gd(N_z/\Q_z))\Disc(\Q_z)^{\ell}=\pm\ell^{\ell(\ell-2)}\N_{\Q_z/\Q}(\f(N_z/\Q_z))^{\ell-1}\;,$$
where $\f(N_z/\Q_z)$ is the conductor.

 By \cite[Theorem 3.7]{CoDiOl3} applied to $K=\Q_z$ and $\al=\eps$ 
which is a unit, we have
$\f(N_z/\Q_z)=(1-\z)^{\ell+1-A_{\eps}}$, where $A_{\eps}=\ell+1$ if
$x^\ell\equiv\eps\pmod{(1-\z)^\ell}$ has a solution in $\Q_z$, and otherwise
$A_{\eps}$ is the maximal $k$ such that $x^\ell\equiv\eps\pmod{(1-\z)^k}$
has a solution. By Lemma \ref{lem:aac_ck} we have $A_{\eps}=(\ell-1)/2$ (resp.,
$A_{\eps}=\ell+1$) if AAC is true (resp., false),
hence $\f(N_z/\Q_z)=(1-\z)^{(\ell+3)/2}\Z_{\Q_z}$ (resp., 
$\f(N_z/\Q_z)=\Z_{\Q_z}$), from which the formula follows (note that the sign 
of the discriminant is positive since $\Q_z$ hence $N_z$ is totally 
complex). In addition, if AAC is false, so that $C_k$ is soluble for
all $k$, then Hecke's Theorem \cite[10.2.9]{Coh1} (an extension of
\cite[Theorem 3.7]{CoDiOl3})
implies that $(1-\z)\Z_{\Q_z}$
is totally split, while if AAC is true then it is totally ramified.\end{proof}

\begin{proof}[Proof of Proposition \ref{prop:special_explicit}]

The result follows for an undetermined $E$ by
Theorem \ref{thm:char_nf} and Proposition \ref{prop:count_sel}. 
To determine $E$, observe that 
Proposition \ref{prop:kum_hec} and the proof of Proposition \ref{prop:count_sel} imply that
$N_z = K_z(\epsilon^{1/\ell})$, and that the considerations in 
the proof of Theorem \ref{thm:char_nf} allow us to take $E$ to be {\itshape any} of the
(conjugate) degree $\ell$ subfields of $N_z$, so that it suffices to exhibit one.

We take $E = \Q(\eps^{1/\ell} - \eps^{-1/\ell})$ for any fixed choice of $\eps^{1/\ell}$, recalling that
the fundamental unit has norm $-1$. 
Then the minimal polynomial of $E$ is $P(x) - \Tr(\eps)$ by construction, or more precisely
by \eqref{eqn:def_cheby}.

It remains only to argue that
$$\Disc(E)=\begin{cases}
\ell^{(3\ell-1)/2}&\text{\quad if AAC is true,}\\
\ell^{\ell-2}&\text{\quad if AAC is false.}\end{cases}
$$
%In addition, in all cases we have
%$\N_{E/\Q}(\gd(N_z/E))=\ell^{\ell-2}=\Disc(\Q_z)$.

We assume that AAC is true (if false, a similar proof applies). On the one hand we have
$$\Disc(N_z)=\Disc(E)^{\ell-1}\N_{E/\Q}(\gd(N_z/E))\;,$$
in other words taking valuations and using the proposition:
$$(\ell-1)v_\ell(\Disc(E))=(3\ell^2-2\ell-3)/2-v_\ell(\N_{E/\Q}(\gd(N_z/E)))
=(\ell-1)(3\ell-1)/2+\ell-2-v_\ell(\N_{E/\Q}(\gd(N_z/E)))\;.$$
On the other hand, the extension $N_z/E$ is of degree $\ell-1$ hence tame,
so $v_\ell(\N_{E/\Q}(\gd(N_z/E)))\le\ell-2$. Divisibility by $\ell-1$
thus implies the result, together with the additional result that
$\N_{E/\Q}(\gd(N_z/E))=\ell^{\ell-2}=\Disc(\Q_z)$.
\end{proof}

We make some additional observations concerning Proposition \ref{prop:special_explicit}:
\begin{remarks}\label{rk:sc}\begin{enumerate}
\item In the equation for $E$ we may replace $\Tr(\eps)$ by
$\Tr(\pm\eps^m)$ for any odd $m\in\Z$ coprime to $\ell$.
\item Assuming AAC, the last product may be written as
$$\left(1-\dfrac{1}{\ell^s}\right)\prod_{p\equiv1\pmod{\ell}}\left(1+\dfrac{\om_E(p)}{p^s}\right)\;.$$
\item When      $p\equiv1\pmod\ell$ then $p$ is totally split in $E$ if and only if
$\eps^{(p-1)/\ell}\equiv1\pmod p$,
and otherwise $p$ is inert in $E$.
\item If in addition $\Q_z=\Q(\z_\ell)$ has class number $1$, then
$p$ is totally split in $E$ if and only if
$p=\N_{\Q_z/\Q}(\pi)$ for some $\pi\equiv1\pmod{\ell}$ in $\Q_z$.
\end{enumerate}\end{remarks}

For (1), it is easily seen that our construction still produces a degree $\ell$
subfield $E$. (2) follows because $\ell$ is totally
ramified in $E$.

To prove (3), again apply \cite[Theorem 10.2.9]{Coh1}: $p$ is totally split in
$E$ iff it is in $N_z/\Q_z$, hence iff $x^\ell\equiv\eps\pmod \p$ is
soluble in $\Q_z$. (Here $\p$ is any prime of $\Q_z$ above $p$, which must have degree $1$ since $p \equiv 1 \pmod \ell$ is totally split in $\Q_z$.)
This is equivalent to $\eps^{(p-1)/\ell}\equiv1\pmod{\p}$, which by Galois theory 
will then be  true for all primes $\p$ above $p$ since for any $\sigma\in\Gal(\Q_z/\Q)$
we have either $\sigma(\eps)=\eps$ or $\sigma(\eps)=-\eps^{-1}$, and $(p-1)/\ell$
is even so the sign disappears. Hence this is equivalent to the condition $\eps^{(p-1)/\ell}\equiv1 \pmod{p}$, as desired.

Finally, (4) follows from Eisenstein's reciprocity law.

\subsection{Explicit computations for $k=\Q$ in the general case}
Let $k = \Q$. In Theorem \ref{thm:char_nf} we saw that 
characters $\chi$ of $G_{\mfb}$ (up to the equivalence $\chi \sim \chi^a$ for
$(a, \ell) = 1$) correspond to degree $\ell$ fields $E$ having certain properties. In our 
companion paper \cite{CRST} with Rubinstein-Salzedo, we further proved the following:

\begin{theorem}\label{thm:nakagawa}\cite{CRST}
Suppose that $k = \Q$ and $K = \Q(\sqrt{D})$ with $D \neq 1, \pm \ell$, so that we are in the general case, and
as before let $K'$ be the mirror field of $K$.

Then the fields $E$ enumerated in Theorem \ref{thm:char_nf} are precisely those $F_{\ell}$-fields $E$
whose Galois closure contains $K'$, subject to the condition $\tau \sigma \tau^{-1} = \sigma^g$ described
there, satisfying the following additional conditions:
\begin{itemize}
\item
$E$ is totally real if $D < 0$, and has $\frac{ \ell - 1}{2}$ pairs of complex embeddings if $D > 0$.
\item
$|\Disc(E)|$ has the form $\ell^{k + b} D^{\frac{ \ell - 1}{2}}$, where $k$ and $b$ satisfy 
\begin{equation}\label{eqn:nakagawa_ab}
\begin{array}{rrl}
k \in \{ 0, 2 \}, & b = \ell - 2 & \text{ if $\ell \nmid D$\;,}\\
k \in \left\{ 0, (\ell + 3)/2 \right\}, & b = \frac{\ell - 3}{2} & \text{ if $\ell \mid D$ and $\ell \equiv 1 \pmod 4$\;,}\\
k \in \left\{ 0, (\ell + 5)/2 \right\}, & b = \frac{\ell - 5}{2} & \text{ if $\ell \mid D$ and $\ell \equiv 3 \pmod 4$\;.}\\
\end{array}
\end{equation}
\end{itemize}
Moreover, if $E$ is any $F_{\ell}$-field satisfying these last two properties, then its Galois closure automatically
contains $K'$.
\end{theorem}
Recall that we have $\mfb \in \calB = \{1, (\ell)^{1/2}, (\ell), (\ell)^{\ell/(\ell - 1)} \}$, with the possibility
$(\ell)^{1/2}$ occurring only if $\ell | D$. The complete list of fields enumerated in Theorem \ref{thm:nakagawa}
corresponds to $\mfb =  (\ell)^{\ell/(\ell - 1)}$. A careful reading of the proof of Theorem \ref{thm:nakagawa}
(in \cite{CRST}), with $k$ having the same meaning above as in \cite[Section 4]{CRST}, shows that the remaining $\mfb$ correspond to the following possibilities for $k$ in \eqref{eqn:nakagawa_ab}:

\medskip
\centerline{
\begin{tabular}{|c||c|c|c|c|}
\hline
Condition on $D$ & $\mfb = 1$ & $\mfb = (\ell)^{1/2}$ & $\mfb = (\ell)$ & $\mfb = (\ell)^{\ell/(\ell - 1)}$\\
\hline
$\ell\nmid D$ & $k = 0$ &  - &   $k = 0, 2$ & $k = 0, 2$ \\
\hline
$\ell\mid D$ and $\ell \equiv 1 \pmod 4$ & $k = 0$ & $k = 0$ & $k = 0, (\ell + 3)/2$ & $k = 0, (\ell + 3)/2$ \\
\hline
$\ell\mid D$ and $\ell \equiv 3 \pmod 4$ & $k = 0$ & $k = 0$ & $k = 0, (\ell + 5)/2$ & $k = 0, (\ell + 5)/2$ \\
\hline
\end{tabular}}
\medskip
One exception occurs for $\ell = 3$: Only $k = 0$ corresponds to $\mfb = (\ell)$ when $\ell \mid D$;
this is because the inequality $(\ell + 5)/2 \leq \ell - 1$ is true for all $\ell \equiv 3 \pmod 4$
except for $\ell = 3$. (Note also for $\ell = 3$ that this result is equivalent to part of Proposition 4.1 in
\cite{Coh-Tho}.)

This is sufficient to obtain an explicit formula for $\Phi_{\ell}(K, s)$ for any $K$ and $\ell$, provided that
the appropriate $F_\ell$-fields can be tabulated. We present two examples, which we also double-checked numerically
using a program written in PARI/GP \cite{pari}.

\begin{example}\label{ex:intro} Let $K = \Q(\sqrt{13})$ and $\ell = 5$. Then, we have

\begin{align*}
\sum_{L\in\calF_5(\Q(\sqrt{13}))}\dfrac{1}{f(L)^s} & =-\dfrac{1}{4}+
\dfrac{1}{20}\left(1+\dfrac{4}{25^s}\right)
\prod_{p}\left(1+\dfrac{4}{p^s}\right)
+ 
\dfrac{1}{5}\left(1 - \dfrac{1}{25^s}\right)
\prod_{p}\left(1+\dfrac{\omega_E(p)}{p^s}\right)
\\
& = 
59^{-s} + 409^{-s} + 475^{-s} + 619^{-s} + 709^{-s} + 1009^{-s} + \cdots + 4 \cdot 24131^{-s} + \cdots
\;,
\end{align*}
where the products are over primes $p \equiv 1, 16, 19, 24, 34, 36, 44, 51, 54, 56, 59, 61 \pmod{65}$,
$E$ is the field defined by the polynomial $x^5 + 5x^3 + 5x - 3$, and $24131 = 59 \cdot 409$.

\end{example}

\begin{example} Let $K = \Q(\sqrt{- 7 \cdot 41})$ and $\ell = 7$. Then, we have
\begin{align*}
\sum_{L\in\calF_7(\Q(\sqrt{-287}))}\dfrac{1}{f(L)^s} & =-\dfrac{1}{6}+
\dfrac{1}{6}\left(1+\dfrac{6}{7^s}\right)
\prod_{p}\left(1+\dfrac{6}{p^s}\right)
+ 
\left(1 - \dfrac{1}{7^s}\right)
\prod_{p}\left(1+\dfrac{\omega_E(p)}{p^s}\right)
\\
& = 
1 + 7 \cdot 301^{-s} + 7 \cdot 337^{-s} + 7 \cdot 581^{-s} + 7 \cdot 791^{-s} + \cdots + 42 \cdot 296897^{-s} + \cdots
\;,
\end{align*}
where the products are over primes $p\equiv\leg{D}{p}\pmod{\ell}$ excluding $p = \ell$,
$E$ is the field defined by the polynomial $x^7 - 14 x^5 + 56 x^3 - 56x - 15$, and 
$296897 = 337 \cdot 881$.
\end{example}

\section{Upper bounds for counting dihedral extensions of $\Q$}\label{sec:upper_bounds}

Write $N_\ell(D_\ell, X)$ for the number of $D_\ell$-fields $L$ with $|\Disc(L)| < X$.
Kl\"uners \cite{kluners} proved that $N(D_\ell, X) \ll X^{\frac{3}{\ell - 1} + \epsilon}$ and in this section
we prove Theorem \ref{thm:improved_bounds}, obtaining $N(D_\ell, X) \ll X^{\frac{3}{\ell - 1} - \frac{1}{\ell (\ell - 1)} + \epsilon}$
as an easy consequence of work of Ellenberg, Pierce, and Wood \cite{EPW}. 

As in the introduction, for each $D_\ell$-field $L$ we have $|\Disc(L)| = n^{\ell - 1} |D|^{\frac{\ell - 1}{2}}$,
where $n \in \Z$ and $\Q(\sqrt{D})$ is the quadratic resolvent field of $D$. For each $D$ and $n$, the multiplicity 
of $L$ with this discriminant is  
$O_\ell(\ell^{\rk_{\ell}(\Cl(D)) + 2 \omega(n)})$; this follows from \cite[Lemma 2.3]{kluners} or alternatively
Theorem \ref{thm_main_q} here.

We therefore have that
\begin{align*}
N(D_\ell, x) \ll & \sum_{|D| < X^{2/(\ell - 1)}} \Cl(D)[\ell] \sum_{n < (X/D^{(\ell - 1)/2})^{1/(\ell - 1)}} \ell^{2 \omega(n)} \\
\\ \ll & \ X^{\frac{1}{\ell - 1} + \epsilon} \sum_{|D| < X^{2/(\ell - 1)}} D^{- \frac{1}{2}} \Cl(D)[\ell].
\end{align*}
By \cite[Theorem 1.1]{EPW}, for all but $O(X^{\frac{2}{\ell - 1} \cdot \left( 1 - \frac{1}{2 \ell} \right)})$ discriminants $D$ in the sum,
we have $\Cl(D) \ll D^{\frac{1}{2} - \frac{1}{2\ell} + \epsilon}$, and the contribution of these
$D$ is
\[
\ll X^{\frac{1}{\ell - 1} + 2\epsilon} \sum_{D < X^{2/(\ell - 1)}} D^{- \frac{1}{2} + \frac{1}{2} - \frac{1}{2\ell}}
\ll X^{\frac{1}{\ell - 1} + 2\epsilon} X^{\frac{2}{\ell - 1} \cdot \left(1 - \frac{1}{2\ell}\right)} = X^{\frac{3}{\ell - 1} - \frac{1}{\ell (\ell - 1)} + 2\epsilon}.
\]
By the trivial bound on $\Cl(D)$, the contribution of the remaining $D$ is also
$\ll X^{\frac{1}{\ell - 1} + 2\epsilon} \cdot X^{\frac{2}{\ell - 1} \cdot \left( 1 - \frac{1}{2 \ell} \right)},$ completing the proof.

\end{document}